\documentclass[11.05pt,a4paper]{amsart}
\usepackage{amssymb,amsmath}
\usepackage{latexsym}
\usepackage{amsthm,amsfonts,amssymb,mathrsfs}
\usepackage{rotating}
\usepackage[leqno]{amsmath}
\usepackage{xspace}
\usepackage[all]{xy}
\usepackage{longtable}
\usepackage{mathrsfs}
\usepackage{amscd,amssymb,amsopn,amsmath,amsthm,graphics,amsfonts,enumerate,verbatim,calc}
\headsep=0.5cm \numberwithin{equation}{section}
\newcommand{\up}[1]{{{}^{#1}\!}}

\newtheorem{theorem}{Theorem}[section]
\newtheorem{lemma}[theorem]{Lemma}\newtheorem{observation}[theorem]{Observation}
\newtheorem{proposition}[theorem]{Proposition}
\newtheorem{corollary}[theorem]{Corollary}\newtheorem{notation}[theorem]{Notation}

\theoremstyle{definition}
\newtheorem{definition}[theorem]{Definition}\newtheorem{fact}[theorem]{Fact}
\theoremstyle{remark}
\newtheorem{remark}[theorem]{Remark}\newtheorem{problem}[theorem]{Problem}
\newtheorem{example}[theorem]{Example}
\newtheorem{question}[theorem]{Question}

\newcommand{\Ass}{\operatorname{Ass}}
\newcommand{\grade}{\operatorname{grade}}

\newcommand{\Spec}{\operatorname{Spec}}

\newcommand{\Tr}{\operatorname{Tr}}

\newcommand{\Ht}{\operatorname{ht}}
\newcommand{\id}{\operatorname{id}}

\newcommand{\Char}{\operatorname{char}}
\newcommand{\pd}{\operatorname{pd}}
\newcommand{\Gdim}{\operatorname{Gdim}}
\newcommand{\Syz}{\operatorname{Syz}}
\newcommand{\CI}{\operatorname{CIdim}}

\newcommand{\Var}{\operatorname{Var}}

\newcommand{\Ext}{\operatorname{Ext}}
\newcommand{\Supp}{\operatorname{Supp}}\newcommand{\Egrade}{\operatorname{E.grade}}

\newcommand{\Tor}{\operatorname{Tor}}
\newcommand{\Hom}{\operatorname{Hom}}
\newcommand{\zd}{\operatorname{zd}}
\newcommand{\DVR}{\operatorname{DVR}}
\newcommand{\Ann}{\operatorname{Ann}}

\newcommand{\depth}{\operatorname{depth}}
\newcommand{\kdepth}{\operatorname{K.depth}}
\newcommand{\Ker}{\operatorname{Ker}}
\newcommand{\Coker}{\operatorname{Coker}}
\newcommand{\im}{\operatorname{im}}

\newcommand{\lo}{\longrightarrow}
\newcommand{\fm}{\frak{m}}
\newcommand{\fp}{\frak{p}}

\begin{document}

\author[]{Mohsen Asgharzadeh and Elham Mahdavi}

\title[ ]{Freeness criteria via   vanishing of $\Tor$ }

\address{M. Asgharzadeh,}
\email{mohsenasgharzadeh@gmail.com}

\address{E. Mahdavi,}
\email{elham.mahdavi.gh@gmail.com}

\subjclass[13D07 ]{Primary: 13D07; 13C10. Secondary: 13D02 }

\keywords{Annihilator; Associated primes; Cohen-Macaulay modules; free modules; Lichtenbaum modules;  Syzygies; $\Tor $-modules. }

\begin{abstract} We investigate some aspects of the module $L$ equipped with the property that $\Tor_1^R(L, F) =0$ implies that $F$ is free. This has some applications.
\end{abstract}

\maketitle
\setcounter{tocdepth}{1}
\tableofcontents
\section{Introduction}

Our initial motivation is to understand the following amusing question of Lichtenbaum:
\begin{question}(See \cite[Question 4]{L})
	For a given local ring $R$, which modules $L$ have the property that
$\Tor_1^R(L, F) =0$ implies that $F$ is free?
	\end{question}
Concerning to the previous item, we call such an $L$ a  $\boldmath{L}$ichtenbaum module.
The classic example  of Lichtenbaum modules  is the residue field. We discuss about the  abundance and basic properties of  Lichtenbaum modules.  
For instance, we detect the following properties of the ring from Lichtenbaum modules:

\begin{enumerate}
	\item[$(i)$]  $\depth_R(R)=0$;
	\item[$(ii)$]   $R$ is a field;
\item[$(iii)$]  $R$ is a $\DVR$;
\item[$(iv)$]   $R$ is regular.
\end{enumerate}

In contrast to Lichtenbaum modules, there are a lot of papers dealing with modules  so called test modules, see \cite{c1} and references therein. Test modules  behave 
both similar to and different from those of Lichtenbaum modules. Our goal, in this paper, is to
present  such  behaviors.
To record a difference, see Proposition \ref{13} and  compare it with \cite[Corollay 3.7]{c1}.

Various aspects of Litchenbaum modules are giving. As a sample, we characterize them
over regular rings, see Corollary \ref{cr}. This not only presents a reverse of \cite[Corollary 6]{L},
but also gives a new  proof of \cite[Corollary 6]{L} and has an advantage over hypersurface rings (see Corollary \ref{cr1}).
It may be worth to mention that, this characterization determines the regularity condition (see Proposition \ref{25}). This has an application in Lazard-type properties. For a sample, see 
Corollary \ref{cr1}.
Also,  this enables
 us to recover a funny and old result of Levin-Vasconcelos by some different arguments. For example,
in the statement of the next result there is no trace of Litchenbaum modules:
\begin{corollary} Let $(R,\frak{m})$  be a local ring. Assume that there exists a positive integer $n$ such that
	$R/\fm^n$ has finite injective dimension.  The following assertions holds:
	\begin{itemize}
		\item[(i)]    $ R $ is Gorenstein.
		\item[(ii)] If $\dim R>0$, then $R$ is regular. \item[(iii)]If $\dim R=0$, then $\fm^n=0$.
	\end{itemize}
\end{corollary}

Also, Lichtenbaum asked:
 
\begin{question}(See \cite[Question 3]{L})
	For which local ring $R$, and for which  module  $T$ have the property that
	$\Tor_i^R(T, N) =0$ implies that $\Tor_{j>i}^R(T, N) =0$?
\end{question}Concerning to the previous item, we call such a module $T$ a tor-rigid module.
 Let $R_0:=k[[x,y]]/(x^2,xy)$, and recall that $R_0$ is of depth zero and dimension one.
In the same paper,  Lichtenbaum  remarked that $R_0/yR_0$ is not tor-rigid.
By using an idea taken from \cite{CT} and \cite{CK}, we present the following connection between Litchtenbaum's questions:
\begin{observation}   Let $(R,\frak{m})$  be a local ring of positive dimension. Then the following are equivalent:
	\begin{itemize}
		\item[(i)]  $\depth_R R>0$.
		\item[(ii)] $R/\fm^n$ is a Lichtenbaum module for $R$ for all integers $n>0$.
		\item[(iii)]  $R/\fm^n$ is  tor-rigid for all $n>0$.
	\end{itemize}
\end{observation}
This may regard as a  converse part
of the  recent works \cite{CT} and \cite{CK}.
We will construct  some nontrivial examples 
 fitting into the setting of the observation.

We present a connection  from Lichtenbaum modules to the Burch modules.
In this regard, we sharpening a celebrated result of Burch, see Observation \ref{b}. This is independent of  \cite{Dey}, and gives
a short and elementary proof of \cite{Dey}. Due to   \cite[4.8+4.7]{h} we know 
 that   integral closure  may be considered as a test module if we restrict ourselves to  the category of finite length modules over a $1$-dimensional complete local
 domain $(R,\fm,k)$ of prime characteristic with $k=\overline{k}$ . In \S 4, we show:

\begin{corollary}\label{c15}
	Let $(R,\frak{m})$ be a $1$-dimensional  local integral domain and let $M$ be        a finitely generated module such that  
	$$\Tor^{R}_{i}(\overline{R},M)=\Tor^{R}_{i+1}(\overline{R},M)=0,$$for some $i>0$. Then
	$\pd_R(M)\leq1$. In particular,   $M$ is free provided it is of positive depth. 	\end{corollary}
It is easy to find examples for which $\overline{R}$ is not tor-rigid, as a sample see Observation \ref{nt}.
Despite this, we present
 a situation for which the vanishing of the corresponding tor is restricted only at a single spot, e.g.,
a situation for which $\overline{R}$ is tor-rigid.
For more details, see Corollary \ref{cr2}.
These lead us to study the length of   syzygies of  Burch modules.
This was asked in \cite{h}: Is $\ell(\Syz_i(M))=\infty$? The illustrative example is as follows.
Recall $R_0$ from the  Lichtenbaum's paper where he proved that
$R_0/yR_0$ is not   Lichtenbaum. Also, we recall from \cite{h} that $\ell(\Syz_2(R_0/yR_0))<\infty$. 
This example  extends as follows:

\begin{observation} 
	Let $L$ be  a Lichtenbaum module such that $\pd_R(L) = \infty$ and $\ell(L) <\infty$.  Then $\ell(\Syz_i(L))=\infty$ for all $i > 0$.
\end{observation}
As an application, we compute the length of syzygies of certain Litchtenbum modules that we constructed in \S3. For a sample, see Corollary \ref{obfr}.

The next topic is about the associated prime ideals of $\Tor_i^R(M, N) $ in terms of $M$ and $N$.
As far as we know, this   rarely computed in the literature. In fact, as far as we know, $\Ass\Tor_0^R(-,\sim)$ is  difficult to compute  even in some
special forms, see \cite{A1} and references therein. 
We determine the associated prime ideals of $\Tor_i^R(M, N)  $  in some  cases. This may  facilitate the study of higher tor-modules. 
For instance, and as an application, we present a situation for which the corresponding $\Tor$-module  is Cohen-Macaulay.
For more details, see   Corollary  \ref{26.5}. Also, we study the annihilator of $\Tor_i^R(M, N)  $  in some nontrivial cases, and  open  some relevant questions.

Finally,
we investigate the vanishing of $\Tor_1^R(-,\sim)$ and dealing with a question asked by Quy and others.
 
 For all unexplained notation and definitions 
 see
 the  books  \cite{BH}  and  \cite{Mat}.
 
\section{Associated primes of $\Tor$}

In this note $(R,\fm,k)$ is a commutative noetherian local ring. The notation $\pd_R(-)$ (resp. $ \id_R(-)$) stands for the projective (resp. injective) dimension.

\begin{question}
	How can compute depth and dimension of $\Tor$-modules? 
\end{question}

We start by recalling the following basic result of Auslander:

\begin{fact}\label{lemma1}(See \cite[Lemma 1.1]{Au})
Let $M$ and $N$ be finitely generated modules  such that $p:=\pd_R M <\infty$ and $\depth_R(N)=0$. Then $\Tor^R_p(M, N) \neq 0$ 
and it is of depth zero.
 \end{fact}
 
\begin{remark}\label{r1}
i)	The finitely generated assumption of $M$ is really important. Here, we follow an idea
	of Lazard \cite{laz}: Let $R:=k[x,y,z]/x\fm$ and  look at $0\to R_{(x)}\stackrel{f}\lo R_{(x,y)}\to M:=\Coker f \to 0$ and we set $N:=k$ which is of depth zero. From this flat resolution, $\Tor_{+}^R(M,N)=0$. Let $t:=\pd_R(M)$. Since $M$ is of finite flat dimension, it follows that $t<\infty$. It remains to note that $\Tor_{t}^R(M,N)=0$.

ii) There is a way to drop the finite assumption of $N$.	In this regard we need to fix the notion
of depth. We say a general module $\mathcal{N}$ is of E-depth zero if $R/\fm \subset \mathcal{N}$.
This holds if any $x$ is zero-divisor over  $\mathcal{N}$. Having this
in mind, the finitely generated assumption of $N$ is not needed. In order to see this, it is enough to apply Auslander's argument. We leave the details to the reader.
\end{remark}
\begin{definition}\label{4} For every $R$-module $L$,  the non-flat locus of $L$ is defined by $$\text{NF}(L):=
\{\frak{p}\in \Spec R \mid \ \text{the}\ R_{\frak{p}}\text{-module}\ L_{\frak{p}} \ {is \ not\ flat}\}.$$ In the case  $L$  is finitely generated,
we call $\text{NF}(L)$ the non-free locus, since over local rings
 a finitely generated flat module is free.  
\end{definition}
We denote the set of all associated prime ideals by $\Ass_R(-)$.
\begin{proposition}\label{5} Let $(R,\frak{m})$ be a local ring and $M$ and $N$ be two  $R$-modules.
If $M$ has finite projective dimension $t\geq 1$, then $\Ass_R(\Tor^R_t(M,N))\subseteq \Ass_RN\cap \text{NF}(M)$. 
\end{proposition}

\begin{proof} Let $\textbf{F}:=0\longrightarrow R^{n_t}\overset{d_t}\longrightarrow R^{n_{t-1}}\overset{d_{t-1}}\longrightarrow \dots
\overset{d_1}\longrightarrow R^{n_0} \longrightarrow 0$ be a free resolution of $M$ where $n_i\in \mathbb{N}\cup\{\infty\}.$ Then
$$\begin{array}{ll}
\Tor_t^R (M,N)&=\text{H}_t(\textbf{F}\otimes_RN)\\
&=\ker (d_t \otimes 1)\\
&\cong \ker ( N^{n_t}\longrightarrow N^{n_{t-1}})\\
&\subseteq N^{n_t},
\end{array}$$
and so $$\Ass_R(\Tor^R_t(M,N))\subseteq \Ass_R(N^{n_t})=\Ass_RN.$$ Since $t\geq 1$,  $\Ass_R(\Tor^R_t(M,N))
\subseteq \text{NF}(M)$. Thus $$\Ass_R(\Tor^R_t(M,N))\subseteq \Ass_RN\cap \text{NF}(M),$$as claimed.
\end{proof}

\begin{corollary}\label{5.5} Adopt the previous assumption and suppose in addition that  $M$ ad $N$ are finitely generated. If
$M$ has projective dimension one, then  $\Ass_R(\Tor^R_1(M,N))=\Ass_RN\cap \text{NF}(M)$.
\end{corollary}

\begin{proof} In view of Proposition \ref{5}, we only need to show that $$\Ass_RN\cap \text{NF}(M)\subseteq  \Ass_R(\Tor^R_1(M,N)).$$ For a finitely
	generated $R$-module $(-)$, it is known that $\fp\in \Ass_R(-)$ if and only if $\depth_{R_{\fp}}((-)_{\fp})=0$. Let $\fp\in \Ass_RN
	\cap \text{NF}(M)$. As $\fp\in \Ass_RN$, we get $\depth_{R_{\fp}}(N_{\fp})=0$. Also, from $\fp\in \text{NF}(M)$ and $\pd_RM=1$,
	we deduce that $\pd_{R_{\fp}}(M_{\fp})=1$. Fact \ref{lemma1}  yields that
	$\depth_{R_{\fp}}(\Tor_1^{R_{\fp}}(M_{\fp},N_{\fp}))=0$, and so $\fp \in \Ass_R(\Tor_{1}^R(M,N))$.
\end{proof}
\begin{notation}
By	$\Syz_{i}(M)$, we mean the $i^{th}$ \textit{syzygy} module of $M$. Some times we set $\Omega (-):=\Syz_{1}(-)$.
\end{notation}	

\begin{corollary}\label{6} Let $(R,\frak{m})$ be a local ring and $M$ and $N$ two finitely generated $R$-modules. Assume that $M$
has finite projective dimension $t\geq 1$. Then $\Ass_R(\Tor^R_t(M,N))=\Ass_RN\cap \text{NF}(\Syz_{t-1}(M))$.
\end{corollary}

\begin{proof} Set $L:=\Syz_{t-1}(M)$. By shifting, $\pd_RL=1$ and $\Tor_t^R (M,N)=\Tor_1^R(L,N)$. Now, apply   Corollary \ref{5.5}.
\end{proof}

\begin{corollary}\label{7} Let $(R,\frak{m})$ be a local ring, $M$ and $N$ be finitely generated. If $M$ has finite
projective dimension $t\geq 1$ and $N$ is $\frak{p}$-primary for some  $\frak{p}\in\Spec(R)$, then $\Ass_R(\Tor^R_t(M,N))
\subseteq \{\frak{p}\}$. In particular, if $\Tor^R_t(M,N)\neq 0$, then $\Ass_R(\Tor^R_t(M,N))=\{\frak{p}\}$.
\end{corollary}
\begin{proof}
The first part is in	Proposition \ref{5}. To see the particular case note that
$$\emptyset  \neq  \Ass_R(\Tor^R_t(M,N))\subseteq\{\frak{p}\},$$ and so $\Ass_R(\Tor^R_t(M,N))=\{\frak{p}\}$.
	\end{proof}
\begin{corollary}\label{8}  Assume that $(R,\frak{m})$ is a local domain and $M$ is a finitely generated $R$-module with
$\pd_RM=1$. Then $\Ass_R(\Tor^R_1(M,M))=\Ass_RM \setminus \{0\}$.
\end{corollary}

\begin{proof}Let $(-)_0$ be the localization with respect to  $0\in\Spec(R)$. Since $M_0$ is an $R_0$-vector space, it follow that $0\notin \text{NF}(M)$. Let $0\neq \frak{p}\in \Ass_RM$. Then $\depth_{R_{\frak{p}}}(M_{\frak{p}})=0$, and so by the Auslander-Buchsbaum formula, we deduce that
$$1\leq \depth (R_{\frak{p}})=\pd_{R_{\frak{p}}}(M_{\frak{p}})+\depth_{R_{\frak{p}}}(M_{\frak{p}})=\pd_{R_{\frak{p}}}(M_{\frak{p}})\leq \pd_RM\leq 1,$$i.e.,  $\pd_{R_{\frak{p}}}(M_{\frak{p}})=1$. In particular, $M_{\frak{p}}$ is not a free $R_{\frak{p}}$-module.
Hence $\text{NF}(M)\cap \Ass_RM=\Ass_RM \setminus \{0\}$, and Corollary  \ref{5.5} completes the proof.
\end{proof}
\begin{corollary}\label{10} Let $(R,\frak{m})$ be a local ring, $M$ and $N$ be  finitely generated. Assume that $M$
is a locally free on the punctured spectrum, and it has finite projective dimension $t\geq 1$. Then $\Tor^R_t(M,N)\neq 0$ if and only if $\depth_RN=0$.
\end{corollary}

\begin{proof} First, assume that $\depth_RN=0$. Fact \ref{lemma1} implies that $\Tor^R_t(M,N)\neq 0$.
Conversely, suppose $\Tor^R_t(M,N)\neq 0$. In view of Proposition  \ref{5} we observe that $$\emptyset \neq \Ass_R(\Tor^R_t(M,N))\subseteq
\text{NF}(M)\cap \Ass_RN=\{\frak{m}\}\cap \Ass_RN.$$ Thus $\frak{m}\in \Ass_RN$, and so $\depth_RN=0$.
\end{proof}

\begin{corollary}\label{26} Let $(R,\frak{m})$ be a complete local ring and $\frak{p}$ a Cohen-Macaulay prime ideal of $R$ of
	finite projective dimension. The following assertions are true:
	\begin{itemize}
		\item[(i)]  $R$ is a Cohen-Macaulay integral domain.
		\item[(ii)] $\Ass_R(\Tor^R_{\Ht \fp}(R/\fp,R/\fp))=\{\fp\}$.
			\item[(iii)] $\dim (\Tor^R_{\Ht \fp}(R/\fp,R/\fp))=\dim R/\fp $.
				\item[(iv)] $\Ann_R(\Tor^R_{\Ht \fp}(R/\fp,R/\fp))=\fp$.
	\end{itemize}
\end{corollary}

\begin{proof} (i) Recall that $R$ is integral domain by intersection theorem. For more details, see \cite{int}. In particular, $R$ is equidimensional
	and catenary. This allows us to apply \cite[Page 250]{Mat} and deduce that$$\dim R-\dim_R(R/\fp)=\Ht\fp=:t.$$  The Auslander-Buchsbaum
	formula implies that
	$$\begin{array}{ll}
	\pd_R(R/\fp)&=\depth_R R-\depth_R(R/\fp)\\
	&=\depth_R R-\dim_R(R/\fp)\\
	&\stackrel{(\ast)}\leq \dim R-\dim_R(R/\fp)\\
	&=\Ht\fp\\
	&=t.
	\end{array}
	$$
	Since $\pd_R(R/\fp)<\infty$, it follows that $\pd_{R_{\fp}}(R_{\fp}/\fp R_{\fp})<\infty$, and so the local ring $R_{\fp}$
	is regular. Now,
	$$\Tor_t^R (R/\fp,R/\fp)_{\fp}\cong \Tor_{\dim_{R_{\fp}}}^{R_{\fp}}(R_{\fp}/\fp R_{\fp},R_{\fp}/\fp R_{\fp})\neq 0.$$
	Consequently, $\Tor_t^R (R/\fp,R/\fp)\neq 0$. So, $\pd_R(R/\fp)=t$. In particular, $(\ast)$
	is an equality, i.e., $\dim R=\depth_R R$.
	
	(ii) Since $\Ass_R(R/\fp)={\fp}$,  and without loss of generality, we may and do assume that $t>0$.  According to Proposition   \ref{5},
	$\Ass_R(\Tor_t^ R(R/\fp,R/\fp))\subseteq \lbrace \fp\rbrace$. As $\Tor_t^R(R/\fp,R/\fp)\neq 0$, it follows that
	$\Ass_R(\Tor_t^ R(R/\fp,R/\fp))\neq \emptyset$, and so $\Ass_R(\Tor^R_{\Ht \fp}(R/\fp,R/\fp))=\{\fp\}$.
	
	(iii) Thanks to part ii) we know    $\Supp(\Tor_t^R(R/\fp,R/\fp)) =\Var(\fp)$. From this,   the desired claim is clear.
	
	(iv) Clearly, $\fp\subseteq\Ann_R(\Tor^R_{\Ht \fp}(R/\fp,R/\fp))$. If the inclusion
	were be strict, we should had $$\dim({R}/{\Ann_R(\Tor^R_{\Ht \fp}(R/\fp,R/\fp))})< \dim R/\fp ,$$because $\fp$ is prime.
	This is in contradiction with the third item. So, $\Ann_R(\Tor^R_{\Ht \fp}(R/\fp,R/\fp))=\fp$,
	as claimed.
\end{proof}

\begin{fact}\label{aus}(See \cite[Theorem 1.2]{Au})
	Let $R$ be any local ring and $\pd_R (M)<\infty$. Let $q$ be the largest number such that  $\Tor_q^R(M, N)\neq0$. If $\depth_R (\Tor_q^R(M, N))\leq1$ or $q=0$,
	then $$\depth_R(N)=\depth_R (\Tor_q^R(M, N))+\pd_R(M)-q.$$
\end{fact}

\begin{corollary}\label{26.5} Adopt  the  assumption  of Corollary \ref{26} and suppose in addition that $\dim R/ \fp\leq2$. Then
 $ \Tor^R_{\Ht \fp}(R/\fp,R/\fp)  $ is   Cohen-Macaulay.
\end{corollary}

\begin{proof} We only  deal with the case $\dim R/ \fp=2$. Thanks to Corollary \ref{26}, 
	$\dim (\Tor^R_{\Ht \fp}(R/\fp,R/\fp))=2 $, and recall that 
	$\depth_R ( \Tor^R_{\Ht \fp}(R/\fp,R/\fp))\leq \dim (\Tor^R_{\Ht \fp}(R/\fp,R/\fp))=2 $.
		Suppose on the way of contradiction that 
	 $ \Tor^R_{\Ht \fp}(R/\fp,R/\fp)  $ is not  Cohen-Macaulay.
Since  $\zd(-)=\bigcup _{P\in \Ass(-)}P$, we observe that any $x\in \fm\setminus \fp$ is a regular sequence over 	 $ \Tor^R_{\Ht \fp}(R/\fp,R/\fp)  $.
It follows that $\depth_R( \Tor^R_{\Ht \fp}(R/\fp,R/\fp)  )=1$. This enables us to apply Fact \ref{aus}. We combine this along with Corollary \ref{26} to see

	$$\begin{array}{ll}
2&=\dim R/\fp\\
&=\depth_R(R/\fp)\\
&=\depth_R ( \Tor^R_{\Ht \fp}(R/\fp,R/\fp))+\pd_R(R/\fp)-{\Ht \fp}\\
&=1+\depth R-\depth_R(R/\fp)-\Ht\fp\\
&=1+\dim R-\dim(R/\fp)-\Ht\fp\\
&=1,
\end{array}
$$which is a contradiction. So,  $ \Tor^R_{\Ht \fp}(R/\fp,R/\fp)  $ is   Cohen-Macaulay.
\end{proof}

It may be nice to recover  $M$ from  $\Tor^R_+(M,M)$. As an easy fact:

\begin{fact}\label{11}  Let $(R,\frak{m})$ be a local ring and $I$ an ideal of $R$ which is generated by an $R$-regular
sequence of length $t$. Then $\Tor^R_i(R/I,R/I)= {\oplus}R/I$ for all $1\leq i\leq t$. In particular,  the following holds. 	\begin{itemize}
	\item[(i)] $\Tor^R_i(R/I,R/I)$ is Cohen-Macaulay iff $R$ is   Cohen-Macaulay.
	\item[(ii)] $\Ass_R(\Tor^R_i(R/I,R/I))=\Ass_R R/I$.\item[(iii)]$\Ann_R(\Tor^R_{i}(R/I,R/ I))=I$.
\end{itemize}  
\end{fact}

\begin{proof} Suppose $I$ is generated by an $R$-regular sequence $\underline{x}:=x_1, \ldots, x_t$, and set $M:=R/I$. The Koszul complex
of $R$ with respect to the sequence $\underline{x}$ has the form
$$0 \longrightarrow R \longrightarrow R^{t} \longrightarrow \dots \longrightarrow R \longrightarrow  0,$$ and it provides a
free resolution of $M$. Thus $$\Tor_i^R(M,M)=H( M^{n_{i+1}} \stackrel{0}\longrightarrow M^{n_i}\stackrel{0}\longrightarrow M^{n_{i-1}})=M^{n_i},$$ and so the desired  claims follow.
\end{proof}

By the linearity, $I\subseteq\Ann_R(\Tor^R_{i}(R/I,R/ I))$. There are many examples
for which the inequality is strict.
A natural question arises: When is $\Ann_R(\Tor^R_{i}(R/I,R/ I))=I$? Let us ask more.
Inspired by Vasconcelos (resp. Simis) \cite{v}, who asked the initial cases,  we ask:

\begin{question}Let $I$ be an ideal and $t>0$.
	When is $\Tor^R_t(R/I,R/I)$ free as an $R/I$-module? 
\end{question}

It may be nice to note that there are 1-dimensional
prime ideals with high number of generators such as $\fp$ in a 3-dimensional regular local rings.
Thus, the assumption of Corollary \ref{26} holds, and  we have
$\Ann(\Tor^R_2(R/\fp,R/\fp))=\fp$ but $\fp$ is not generated be a regular sequences.
However, $\Tor^R_2(R/\fp,R/\fp)$ is not free  as an $R/\fp$-module.

\section{Freeness criteria via vanishing of $\Tor_{1}$}

We start with:
\begin{definition}\label{1} A  nonzero $R$-module $L$ is called (strong)  Lichtenbaum if for every (not-necessarily)  finitely generated $R$-module $F$, the vanishing of
	$\Tor^R_1(L,F)$ implies that $F$ is (flat) free.
\end{definition}

\begin{remark}
 In most of the applications,  we assume $F$ (resp. $L$) is finitely generated. Since the ring is local,
any finitely generated flat module is free.\footnote{for the case $F$ (resp. $L$)  is not-necessarily  finitely generated,    see Remark \ref{23} (resp. Proposition \ref{15}).}
\end{remark}

\begin{lemma}\label{tf}
 Let $(R,\frak{m})$ be a local ring and let $L$ be a module such that
  $$\Tor^{R}_{1}(L,F)=0\Longrightarrow\emph{F is faithful}.$$
Then $\depth_R(L)=0$.
\end{lemma}

\begin{proof}
	On the contrary, assume that $\depth_RL>0$. Then there is an $L$-regular element $x\in \frak{m}$.
 Consider the free resolution 	of the $R$-module $R/xR$:
	$$\xymatrix{&\cdots \ar[r]&R^n\ar[r]^{}\ar[d]_{_{\pi}}&R\ar[r]^{x}&R\ar[r]^{}&R/xR\ar[r]^{}&0\\
		&& (0:_Rx)\ar[ur]_{\subseteq}\ar[d]\\
		& 0\ar[ur]& 0
		&&&}$$
 By definition, $$\Tor^{R}_{1}(L,R/xR)=\text{H}(L^n\longrightarrow L\overset{x}\longrightarrow L)=0.$$
According to the assumption, it follows that $R/xR$ is faithful, which is our desired contradiction.
	\end{proof}

\begin{remark}
	In Lemma \ref{tf}
	one can not replace the faithful assumption with the torsion-free. Indeed, let $(R,\fm)$
	be of depth zero and of positive dimension.
This gives us a non maximal prime ideal $\fp$.	Let $L:=R/  \fp$. Since  $\depth(R)=0$,
any module is torsion-free. In particular, the following implication holds: $$\Tor^{R}_{1}(L,F)=0\Longrightarrow\emph{F is torsion-free}.$$In order to see the left hand side is not empty, let $F:=R/xR$ where $x\in\fm\setminus \fp $. It is easy to see $\Tor^{R}_{1}(L,F)=0$. 
But, $\depth(L)>0$, because $x$ is regular over it.
\end{remark}

\begin{corollary}\label{12}  Let $(R,\frak{m})$ be a local ring  and $L$ a finitely generated
	Lichtenbaum module for $R$. Then
  $\depth_RL=0$.
\end{corollary}

\begin{proof}   
	This is in Lemma \ref{tf}.
\end{proof}

\begin{remark}\label{rng}
Adopt the notation of Corollary \ref{12}.  The finitely generated assumption is not so important. In fact, we need to fix the notion of depth. This is  defined to the supremum
of the lengths of all weak regular sequences on $L$. This is called the classical depth. Now, Corollary \ref{12} can be extend, by the same proof, as follows: Let $(R,\frak{m})$ be a local ring  and $L$ a 
Lichtenbaum module for $R$. Then
the classical depth of $L$ is zero.
\end{remark}

 In one of the applications we deal with regular rings and we  need to use an extension of Remark \ref{rng}. Suppose $R$ is regular. In order to extend the previous item,
  we recall the concept of  Koszul depth: Suppose  $\fm$ is
  generated by a set
$\underline{x}:=x_{1},\cdots, x_{d}$, we denote the Koszul
complex of $R$ with respect to $\underline{x}$ by
$\textbf{K} (\underline{x})$. Koszul depth of $\fm$ on $M$
is defined by
$$\kdepth_R(M):=\inf\{i \in\mathbb{N}\cup\{0\} | H^{i}(\Hom_R(
\textbf{K} (\underline{x}), M)) \neq0\}.$$ Note that by
\cite[Corollary 1.6.22]{BH} and \cite[Proposition 1.6.10 (d)]{BH},
this does not depend on the choice of generating sets of $\fm$. Now,
Corollary \ref{12} extends in the following sense:

\begin{corollary}\label{exte}
Let $(R,\frak{m})$ be a Cohen-Macaulay local ring  and $L$ be a
Lichtenbaum module for $R$. Then
$\kdepth_R(L)=0$.
\end{corollary} 

\begin{proof}
	 Let $d:=\dim R$. We may assume that $d>0$, because 
	 $\kdepth_R(L)\leq d$.
	Suppose on the way of contradiction that $\kdepth_R(L)>0$. 
Since $R$ is Cohen-Macaulay,   $\fm$ is a radical an ideal
generated by a set
$\underline{x}:=x_{1},\cdots, x_{d}$
of length $d$ which 
  is a regular
sequence. 
In this case, the Koszul
complex with respect  to $\underline{x}$ is a free resolution of $R/ \underline{x}R$. We know that Koszul grade is unique up to radical.
So, Koszul grade of $L$ on $\underline{x}R$ is also zero.
Also, recall that the symmetry of Koszul cohomology
and Koszul homology says that
$$H_{i}(\textbf{K}(\underline{x})\otimes_RL)\cong
H^{d-i}(\Hom_R(\textbf{K}(\underline{x}),L)),$$ see
\cite[Proposition 1.6.10 (d)]{BH}.
Thus,
	$$\Tor^R_{1}(\Syz_{d-1}(R/\underline{x}R),L)\cong\Tor^R_{d}(R/\underline{x}R,L)=H_d\left(\textbf{K}(\underline{x},R)\otimes_RL\right)=H^{0}(\Hom_R(\textbf{K}(\underline{x}),L))=0.$$Since $L$
	is Lichtenbaum, we deduce that $\Syz_{d-1}(R/\underline{x}R)$ is free. In other words, $\pd_R(R/\underline{x}R)\leq d-1$.
By Auslander-Buchsbaum formula,  $\pd_R(R/ \underline{x}R)=d$. This contradiction completes the proof.
\end{proof}

\begin{notation}
	The notation $\Tr(-)$ stands for the Auslander's transpose.
\end{notation}

\begin{lemma}\label{123}  Let $(R,\frak{m})$ be a local ring  which isn't a field and $L$ a finitely generated
Lichtenbaum module for $R$. The following assertions are valid:
\begin{itemize}
\item[(i)]  $\pd_RL<\infty$ if and only if $R$ is regular.
\item[(ii)] $L$ is not e-rigid, i.e. $\Ext_R^1(L,L)\neq 0$.

\end{itemize}
\end{lemma}

\begin{proof} (i)  Assume that $d:=\pd_RL<\infty$. Then $$0=\Tor^{R}_{d+1}(R/\frak{m},L)=\Tor^{R}_{1}(L,\Syz_{d}(R/\frak{m})).$$
As $L$ is  Lichtenbaum, this implies that  $\Syz_{d}(R/\frak{m})$ is free, and so $\pd_R
R/\frak{m}<\infty$. Hence, $R$ is regular.
The reverse implication is obvious.

(ii) We suppose that $L$ is e-rigid, and look for a contradiction. Recall that $\Omega (-):=\Syz_{1}(-)$. From the exact
sequence (see \cite[2.8]{AB})$$\Tor^{R}_{2}(\Tr(\Omega L),L)\longrightarrow \Ext^{1}_{R}(L,R)\otimes_RL\longrightarrow \Ext^{1}_{R}(L,L)
\longrightarrow \Tor^{R}_{1}(\Tr(\Omega L),L)\longrightarrow 0,$$ we conclude that $\Tor^{R}_{1}(\Tr(\Omega L),L)=0$.
As $L$ is a Lichtenbaum module for $R$, we deduce that the $R$-module
$\Tr(\Omega L)$ is free. Recall that $\Tr(\Tr (-))\cong (-)$, and so $\Omega L\cong \Tr(\Tr(\Omega L)$ is free.
Thus, $\pd_RL\leq 1$. Next, we have $$\Ext_{R}^{1}(L,R)\otimes_RL\cong \Ext_{R}^{1}(L,L)=0.$$ This yields that
$\Ext^{1}_{R}(L,R)=0$, and so $$\pd_RL=\sup \{n\in\mathbb{N}_0\mid \Ext^{n}_{R}(L,R)\neq 0\}=0.$$  This is a
contradiction because a free module can't be   Lichtenbaum   unless the ring is field, and this is excluded from the assumption.
\end{proof}

\begin{notation}Suppose  $p:=\Char (R)>0$, and let $\varphi:R\to R$ denotes the  Frobenius endomorphism given by $\varphi(a)=a^{p}$ for $a\in R$. Each iteration $\varphi^n$ of $\varphi$ defines a new $R$-module structure on the set $R$, and this $R$-module is denoted by $\up{\varphi^n}R$, where $a\cdot b = a^{p^{n}}b$ for $a, b \in R$.
\end{notation}

We present some (non-) examples of Lichtenbaum modules.

\begin{example}\label{none}Let $(R,\fm,k)$ be local. The following holds:
	\begin{itemize}
\item[(i)] Let $x\in \fm$ be an $R$-regular element. Then $\fm/x\fm$ is a Lichtenbaum module for $R$.
\item[(ii)]  $ \fm $ is Lichtenbaum if and only if $ \depth_R R=0 $.
\item[(iii)] Suppose  $R$ is Cohen-Macaulay and  $p:=\Char (R)>0$.  
Then $\up{\varphi^n}R$ is Litchtenbaum for all $n\gg 0$ iff $\dim(R)=0$.
\item[(iv)] Localization of a Lichtenbaum module is not necessarily Lichtenbaum, even over $2$-dimensional (regular) rings.
\end{itemize}
\end{example}

\begin{proof}
(i) Suppose $\Tor^{R}_{1}(\fm/x\fm,F)=0$ for some finitely generated $R$-module $F$. We are going to show $F$ is free.
Indeed, since $x$ is regular over $R$, it is also regular over $\frak{m}$. So, we have the short exact sequence
$$0\longrightarrow \frak{m}\overset{x}\longrightarrow \frak{m}\longrightarrow \frak{m}/x\frak{m}\longrightarrow 0.$$
It yields the exact sequence $$\Tor^{R}_{1}(\frak{m},F) \overset{x}\longrightarrow \Tor^{R}_{1}(\frak{m},F)
\longrightarrow \Tor^{R}_{1}(\frak{m}/x\frak{m},F)=0.$$ By Nakayama's lemma, we conclude that $\Tor^{R}_{1}(\frak{m},F)=0$.
By shifting,  
we have $\Tor^{R}_{2}(R/\frak{m},F)\cong \Tor^{R}_{1}(\frak{m},F)=0$, and so $\pd_RF\leq1$.
It is easy to see that the map $\psi: R/\frak{m}\longrightarrow \frak{m}/x\frak{m}$ defined by $\psi(r+\frak{m})=xr+x\frak{m}$
is an injective $R$-homomorphism. This fits in the following   short exact sequence $$0 \longrightarrow R/\frak{m} \overset{\psi}\longrightarrow \frak{m}/x\frak{m} \longrightarrow \Coker \psi \longrightarrow0,$$ which yields the following exact sequence:
$$\Tor^{R}_{2}(\Coker \psi, F)\longrightarrow \Tor^{R}_{1}(R/\frak{m},F)\longrightarrow \Tor^{R}_{1}(\frak{m}/x\frak{m},F).$$
Hence $\Tor^{R}_{1}(R/\frak{m},F)=0$, and so $F$ is free.

(ii) It is enough to apply Corollary \ref{12}  along with the argument presented in  part (i).

(iii) Let $d:=\dim R$. First assume that $d=0$, and suppose $\Tor^{R}_{1}(\up{\varphi^n}R,F)=0$ for some finitely generated $R$-module $F$. Denote the  n-th Frobenius power of an ideal  $(-)$ by $(-) ^ { [p^n]}  $.
Let $n$ be  such that 
$\fm ^ { [p^n]}   =0$. This is possible, because $\fm$ is nilpotent. 
In other words, $\fm. \up{\varphi^n}R =0$.  Then $\oplus\Tor^{R}_{1}(k,F)=\Tor^{R}_{1}(\up{\varphi^n}R,F)=0$. Since
$k$ is a test module and $F$ is finitely generated, 
we deduce that $F$
is free. So,  $\up{\varphi^n}R$ is Litchtenbaum.
Now, assume $d>0$ and suppose on the way of contraction that
$\up{\varphi^n}R$ is Litchtenbaum. According to Remark \ref{rng}  we know that the
classical depth of $\up{\varphi^n}R$ is zero. In view of  \cite[Theorem 16.1]{Mat} we observe
that $\depth(R)=0$. This is in contradiction with the Cohen-Macaulay assumption.

(iv) Let $R$ be any  local ring of dimension at least two, e.g., $R:=k[[x,y]]$ and look at the following
chain $\fp_1\subsetneqq\fp_2\subsetneqq\fm$ of prime ideals. Let $L:=R/\fm \oplus R/\fp_1$. As $k$ is a direct summand of $L$, we observe that $L$ is Litchtenbaum. Since $L_{\fp_2}$ is of positive depth
over $R_{\fp_2}$, and due to Corollary \ref{12}, we deduce that  $L_{\fp_2}$ is not Litchtenbaum
over $R_{\fp_2}$.
\end{proof}
In \S5 we will   apply  Example \ref{none}(iii) for not necessarily Cohen-Macaulay rings. Let us drop the Cohen-Macaulay assumption:
\begin{remark}\label{fr}
	 Suppose $R$ is local and $p:=\Char (R)>0$.  
	 Then $\up{\varphi^n}R$ is Litchtenbaum for all $n\gg 0$ iff $\depth(R)=0$.\end{remark}
\begin{proof}	Indeed, this is similar to Example \ref{none}(iii). It is enough to
	 combine   Remark \ref{rng} along with \cite[Proposition 2.2.11]{mi}.  
\end{proof}
\begin{definition}(Auslander)
	A module $ T $ is said to be tor-rigid if there is a non-negative integer $n$ such that for every   finitely generated   
$R$-module $M$, vanishing of $\Tor_{n}^R(T,M)$ implies $\Tor_{n+i}^R(T,M)$ vanishes for all $i\geq 0$.
\end{definition}

\begin{remark}
	i) Over regular rings any finitely generated  module is tor-rigid (see \cite[2.2]{Au}). 
	
	ii)  The finiteness assumption  is important.
	Indeed, let $(R,\fm)$ be a regular local ring of dimension $d>1$. By local duality,
	$\Tor^R_1(E_R(k),k)=0$ and $\Tor^R_d(E_R(k),k)=k$. 
\end{remark}

The if part of the next result is in \cite[Corollary 6]{L}. We recover it by a new argument:
\begin{corollary}\label{cr}
	Let $(R,\fm)$ be a regular and local ring, and let $L$ be finitely generated. Then $L$ is Lichtenbaum if and only if $\depth_R(L)=0$.
\end{corollary}

\begin{proof}
		Suppose $\depth_R(L)=0$ and $\Tor^R_1(L,F)=0$  for some finitely generated $R$-module $F$. By   tor-rigidity, we observe that $\Tor_+^R(L,F)=0$.
			By depth formula (see Fact \ref{aus}), $$\depth_R(R)\geq\depth_R(F)=\depth_R(L\otimes F)+\pd_R(L)=\depth_R(L\otimes F)+\depth_R R\geq \depth_R R.$$
			This yields that $\depth_R(F)=\depth_R(R)$. By Auslander-Buchsbaum formula, $F$ is free. Thus, $L$ is Lichtenbaum. The reverse part is in  Corollary  \ref{12}.
	\end{proof}
	\begin{corollary}\label{cr1}
		Let $(R,\fm)$ be a regular and local ring, and let $L$ be  Lichtenbaum.
	Then $L$ is a directed limit of finitely generated  Lichtenbaum modules.
	\end{corollary}
	
	\begin{proof}
		Recall from Corollary \ref{exte} that $\kdepth_R(L)=0$. It turns out that
		$$\Egrade_{R}(\fm,L):=\inf\{i\in \mathbb{N}\cup\{0\}|\Ext^{i}_{R}(R/\fm,
		L)\neq0\}=0.$$Let $f:R/\fm\to L$ be any nonzero morphism, and recall that $R/\fm$ is simple as an $R$-module. From these, $R/\fm$ can be embedded into $L$.
		There is a filtered system $\{L_i\}_{i\in I}$ of finitely generated submodules of $L$
		such that $L=\bigcup_{i\in I} L_i$ and that $R/ \fm\subseteq L_i$.
		In particular, $\depth_R(L_i)=0$. In view of Corollary \ref{cr} we observe that
		$L_i$ is  Lichtenbaum, and this completes the proof.
	\end{proof}
	\begin{notation}
	  Let $(-)^{\vee}:=\Hom_{R}(-,E_R(R/\frak{m}))$ be the Matlis duality  functor.
	\end{notation}
Also, there is a connection to   field theory:
 
\begin{proposition}\label{13}  Let $(R,\frak{m})$ be a local ring. The following are equivalent:
	
	\begin{itemize}
		\item[(i)]   $R$ is a field.
		\item[(ii)] $R$ is Cohen-Macaulay with canonical module	$\omega_R$ which  is  Lichtenbaum. 
		\item[(iii)]$K_R$  is a Lichtenbaum module for $R$.		\end{itemize}
\end{proposition}

\begin{proof}(i) $\Rightarrow$ (ii) and (ii) $\Rightarrow$ (iii): These are easy.
	
	(iii) $\Rightarrow$ (i): Recall that $K_R=H^d_{\fm}(R)^{\vee}$ where  $d:=\dim R$.  As $K_R$ is  Lichtenbaum, 
and in view of  Corollary  \ref{12}, $\depth_R(K_R)=0$.	Since it satisfies Serre's condition $(S_2)$, it follows that $d=0$. In particular, $K_R=\omega_R$.
 Let $M$ be   finitely generated and torsionless as an $R$-module.
As $M$ is torsionless, from the exact sequence (see \cite[2.6]{AB}) $$0\longrightarrow \Ext_{R}^{1}(\Tr M,R)\longrightarrow M\overset{\varphi}
{\longrightarrow} M^{**}\longrightarrow \Ext_{R}^{2}(\Tr M,R)\longrightarrow 0,$$ we deduce that $\ker \varphi=0$,
and so $\Ext_{R}^{1}(\Tr M,R)=0$. 
Now,
$$\begin{array}{ll}
0&=\Ext_{R}^{1}(\Tr M,R)^{\vee}\\
&\cong \Tor^{R}_{1}(\Tr M,E(R/\frak{m}))\\
&=\Tor^{R}_{1}(\Tr M, \omega_R)\\
&\cong \Tor^{R}_{1}(\omega_R,\Tr M).
\end{array}$$

As $\omega_R$ is  Lichtenbaum, it follows that $\Tr(M)$ is free. In order to see $M$ is free, we apply
another Auslander's transpose and obtain $M\cong \Tr(\Tr(M))$ which is free.
Next, let $M$ be an arbitrary  finitely generated $R$-module. Consider the short exact sequence
$$0\longrightarrow \Syz_{1}(M) \longrightarrow R^{n} \longrightarrow M\longrightarrow 0.$$

As $\Syz_{1}(M)$ is a submodule of a free $R$-module, it is torsionless. We repeat  the above argument, and deduce that
the $R$-module $\Syz_{1}(M)$ is free. In other words, $\pd_RM\leq 1$. Now,  Auslander-Buchsbaum formula asserts
that $\pd_RM=0$. In sum, we proved that any finite generated $R$-module is free, and so $R$ is a field.
\end{proof}

\begin{corollary}\label{14}  Let $(R,\frak{m})$ be an artinian local ring which is not a field and $L$ be a finitely
generated Lichtenbaum module for $R$. Then $L$ has infinite injective dimension.
\end{corollary}

\begin{proof}  On the contrary, suppose $\id_RL<\infty$. 
 It follows that $\pd_R(L^{\vee})<\infty $. By  Auslander-Buchsbaum formula, we
see that $\pd_R(L^{\vee})=0$. Since the ring is local, $L^{\vee}$ is free. Now, $L\cong L^{\vee \vee}$ is injective. Due to the  Matlis
theory, we know $$L \cong \oplus E_{R}(R/\fm) \cong \oplus \omega_R,$$ and so $\omega_R$ is a Lichtenbaum module for
$R$. Now, in view of Proposition  \ref{13}, we arrived at the desired contradiction.
\end{proof}

\begin{proposition}\label{15}  Let $(R,\frak{m})$ be a local domain which is not a field and $L$ a finitely generated Lichtenbaum
module for $R$. Then $\Omega L\otimes_R\Omega L$ is torsion-free if and only if $R$ is a discrete valuation ring.
\end{proposition}

\begin{proof} First, assume that $R$ is a discrete valuation ring.  Then $\text{gdim} R=1$, and so $\Omega N$ is a free
module for all $R$-modules $N$. In particular, it follows that  $\Omega L\otimes_R \Omega L$ is free, and
clearly, it is torsion-free.
Conversely, suppose $\Omega L\otimes_R\Omega L$ is torsion-free. From the  exact sequence
$$0\longrightarrow \Omega L \longrightarrow R^{n} \longrightarrow L \longrightarrow 0,$$ we deduce the exact sequence
$$0\longrightarrow \Tor_{1}^{R}(L,\Omega L)\longrightarrow \Omega L\otimes_R\Omega L\longrightarrow (\Omega L)^{n}
\longrightarrow L\otimes_R\Omega L \longrightarrow 0  \   \        \ (*)$$
Set $S:=R\setminus 0$. As $S^{-1}R$ is a field, it follows that $$S^{-1}(\Tor_{1}^{R}(L,\Omega L))\cong
\Tor_{1}^{S^{-1}R}(S^{-1} L,S^{-1}(\Omega L))=0.$$ Hence, the $R$-module $\Tor_{1}^R(L,\Omega L)$ is torsion.
As $\Omega L\otimes_R\Omega L$ is torsion-free, from the exact sequence $(*)$, we conclude that $\Tor_{1}^R(L,\Omega L)=0$.
Since $L$ is a Lichtenbaum module for $R$, it turns out that $\Omega L$ is free, and this implies that $\pd_RL\leq 1$. Thus, by Lemma \ref{123}(i),
it turns out that $R$ is regular. Now, the Auslander-Buchsbaum formula and  Corollary  \ref{12}  imply that $\dim R\leq 1$.
As $R$ is not a field, we get that $\dim R=1$, and so $R$ is a discrete valuation ring.
\end{proof}

We observed in Example \ref{none}(ii) that certain modules are  Lichtenbaum iff the ring is of depth zero.  In the same vein, we  detect the $\DVR$ property of rings. 

\begin{observation}\label{16}  Let $(R,\frak{m})$ be a Gorenstein local ring which is not a field. Then $E_R(R/\fm)$
is a Lichtenbaum module for $R$ if and only if $R$ is a discrete valuation ring.
\end{observation}

\begin{proof} First, assume that $E:=E_R(R/\fm)$ is a Lichtenbaum module for $R$. Let $d:=\dim R$ and $\underline{x}=x_1,
\dots, x_d$ be a system of parameters for $R$. Recall that the \v{C}ech complex of $R$ with respect to $\underline{x}$
has the form $$\check{C}:=0\longrightarrow R \longrightarrow \oplus R_{x_i} \longrightarrow \dots \longrightarrow \oplus
R_{x_1\dots \hat x_i \dots x_d}\longrightarrow R_{x_1 \dots x_d} \longrightarrow 0,$$ and it provides a flat resolution
for the $R$-module $\text{H}_{\frak{m}}^d(R)$. Hence, 
$$\Tor_i^R(E,M)\cong \Tor_i^R(\text{H}_{\frak{m}}^d(R),M)\cong \text{H}_i(\check{C}\otimes_RM)\cong \text{H}_{\frak{m}}^{d-i}(M)$$
for every $R$-module $M$.
If $d\geq 2$, then $\Tor_1^R(E,R/\fm)\cong \text{H}_{\frak{m}}^{d-1}(R/\fm)=0$, while $R/\fm$ is not a free $R$-module.
So, $d\leq 1$. If $d=0$, then Corollary \ref{14} indicates that $\id_RE=\infty$, which is a contradiction. Thus $d=1$. Next,
we have $$\Tor_1^R(E,\Syz_1(R/\fm))\cong \text{H}_{\frak{m}}^0(\Syz_1(R/\fm))=0.$$Since $L$ is  Lichtenbaum, we deduce that $\Syz_1(R/\fm)$ is
free. This in turns is equivalent with  $\pd_R(k)\leq 1$. In other words, $R$ is a discrete valuation ring, because the ring is not a field.

Conversely, assume that $R$ is a discrete valuation ring. Let $F$ be a finitely generated $R$-module such that
$\Tor_1^R(E,F)=0$. We need to show  $F$ is free. As $\text{H}_{\frak{m}}^0(F)\cong \Tor_1^R(E,F)=0$, we obtain 
$\depth_RF\geq 1$. Over a discrete valuation ring, every module has finite projective dimension. Now, the Auslander-Buchsbaum
formula yields that $F$ is free.
\end{proof}

\begin{notation}
i) 	By $\mu(-)$ we mean the minimal number of elements that  needs to generates a finitely generated  module $(-)$.

ii) By $\ell(-)$ we mean the length function.

iii) Here, $e(R)$
is the Hilbert-Samuel multiplicity. 
\end{notation}

Let us connect to the multiplicity.

\begin{observation}\label{19} Let $(R,\fm)$ be a $1$-dimensional local ring of depth zero. If $e(R)=1$, then $R/\fm^n$ is not a Lichtenbaum module  for all $n\gg0$.
\end{observation}

\begin{proof}  By definition,  $\ell(R/\fm^n)=e(R)n-c=n-c $  for all $n\gg 0$, where $c$ is a constant. Hence, $$1=(n+1)- c - (n-c) =
\ell(R/\fm^{n+1}) - \ell(R/\fm^n) = \ell(\fm^n/\fm^{n+1})=\mu (\fm^n),$$where
the last  (resp. the third)  equality is in \cite[Theorem 2.3]{Mat} (resp.  follows from the short exact sequence $0\to\fm^n/\fm^{n+1}\to  R/\fm^{n+1} \to R/\fm^n\to 0$).   
Say $ \fm^n = xR $ for some $  x \in R $ and for all $ n\gg 0 $. 
In particular, $ x $ is a system of parameter. Also, $ x $ is not nilpotent. 
There is a minimal prime ideal $ \fp$ such that $ x \notin \fp $. This implies that the multiplication map $ R/\fp \stackrel{x}\longrightarrow R/\fp$ is injective. 
Now, we look at
$$\xymatrix{& &R^n\ar[r]^{}\ar[d]_{\pi}&R\ar[r]^{x}&R\ar[r]& R/xR \ar[r]^{} &0\\
&& (0:_Rx)\ar[ur]_{\subseteq}
&&&}$$where $\pi$ is the natural epic.
We apply $ - \otimes R/\fp $ to the displayed exact sequence and deduce the following:

	\begin{equation*}
\begin{array}{clcr}
\Tor_{1}^R (R/\fm^{n} ,  R/\fp ) &= \Tor_{1}^R (R/xR, R/\fp)\\&=H ((R/\fp)^n \overset{\varphi}\longrightarrow R/\fp \overset{x}\longrightarrow R/\fp)\\ &=\frac{\Ker (R/\fp \stackrel{x}\longrightarrow R/\fp)}{ \im (\varphi)}\\
&
=0.
\end{array}
\end{equation*}
If $ R/\fp $ were be free, then we should had $\fp=0$, this is impossible, because $\depth(R)=0$.
Since $ R/\fp $ is not free, $ R/\fm^n $ is not Lichtenbaum.
\end{proof}

Here, we extend the previous observation to the general setting. 
\begin{proposition}\label{20}  Let $(R,\frak{m})$  be a local ring which is not artinian. Then the following are equivalent:
\begin{itemize}
\item[(i)]  $\depth_R R>0$.
\item[(ii)] $R/\fm^n$ is a Lichtenbaum module for $R$ for all integers $n>0$.
\item[(iii)]  $R/\fm^n$ is  tor-rigid for all $n>0$.
\end{itemize}
\end{proposition}

\begin{proof} (i) $\Rightarrow$ (ii):  If $R$ is regular, then the assertion follows by Corollary \ref{cr}.
	So, without loss of generality, we may assume that $R$ is not regular. Now, \cite[Corollary 2.14]{CK} yields the claim.

(ii) $\Rightarrow$ (i): Suppose $R/\fm^n$ is a Lichtenbaum module for all integers $n>0$. We are going to show $\depth_R R>0$.
On the contrary, assume that $\depth_R R=0$. Then $ \fm \in \Ass_RR$. By definition, there is some nonzero element $x$ of $R$ such that $\fm=(0:_Rx)$.
Since $R$ is not a field, we get that $x\in \fm$. By Krull's intersection theorem, we have $\underset{n\in\mathbb{N}}
\bigcap \fm^n =0$. As $x\neq 0$, there is a natural number $n$ such that $x \notin \fm^n $.

Let $y$ be  in $\langle x \rangle\cap \fm^n$. Take $r\in R$
be such that $y=rx$. We have two possibilities: Either $r\in \fm$ or $r\notin \fm$. In the first case  $rx= 0$,
and so $y=0$. Now, suppose that $r\notin \fm$. As the ring is local, $r$ should be a unit. From this, $x=r^{-1}y\in \fm^n$. By the
choice of $n$, we get to a contradiction. In sum, we proved that $\langle x \rangle\cap \fm^n =0$. In particular, $$\Tor_{1}^R (R/\fm^n,R/\langle x \rangle)=(\langle x
\rangle\cap \fm^n)/x \fm^n =0.$$ Now, recall that  $R/\langle x \rangle$ is not   free. In view of Definition \ref{1}  we conclude that  $R/\fm^n$ is not Lichtenbaum.
This contradiction shows that $\depth_R R>0$.

(i) $\Rightarrow$ (iii): This holds by \cite[Corollary 1.3]{CT}.

(iii) $\Rightarrow$ (i):
On the contrary, suppose $\depth_R R=0$, it follows  $\fm=(0:_Rx)$ for some nonzero $x\in R$. By Krull's
intersection theorem,   there is a natural number $n$ such that $x\notin \fm^n $. In the previous
argument, we observed  that $\Tor_{1}^R (R/\fm^n,R/\langle x \rangle)=0.$ As $R/\fm^n$ is tor-rigid, it turns out that
$\Tor_{2}^R(R/\fm^n,R/\langle x \rangle)=0$. From
$0 \to \langle x \rangle \to R \to R/\langle x \rangle \to 0$
we deduce that $$\Tor_{1}^R(R/\fm^n,\langle x \rangle)\cong\Tor_{2}^R (R/\fm^n,R/\langle x \rangle) \quad(+)$$
Also, $0 \to  (0:_Rx) =\fm   \to R \stackrel{x} \to \langle x \rangle \to 0$  yields that $$R/ \fm \cong \langle x \rangle\quad(\ast)$$
Combining these together: $$\Tor_{1}^R(R/\fm^n,R/\fm)\stackrel{(\ast)}\cong  \Tor_{1}^R (R/\fm^n ,\langle x \rangle)\stackrel{(+)}\cong
\Tor_{2}^R (R/\fm^n,R/\langle x \rangle)=0.$$

As $R/\fm$ is Lichtenbaum, it follows that $R/\fm^n$  is  free.
By
taking annihilator, $\fm^n=0$. This is excluded by the assumptions, a
 contradiction. So, the desired claim follows.
\end{proof}
\begin{remark} Adopt the above assumption.
\begin{itemize}
	\item[(a)]  Proposition \ref{20} may be considered as  the reverse parts of \cite[Corollary 1.3]{CT} and  \cite[Corollary 2.14]{CK}.

	\item[(b)] The non-artinian assumption is important. For example, let $R:=k[x]/(x^2)$. It is easy to see  $R/\fm^n$ is  tor-rigid for all $n>0$. But, $\depth_R(R)=0$.
	\item[(c)] In on other to see the item $b)$ is special, let $R:=k[x]/(x^n)$ where $n>2$. It is easy to see  $R/\fm^{n-1}$ is not tor-rigid. Indeed,
	let $M:=R/x^{n-1}R$, and look at its free resolution $$\ldots\lo R\stackrel{x}\lo R\stackrel{x^{n-1}}\lo R\stackrel{x}\lo\ldots  \stackrel{x}\lo R\stackrel{x^{n-1}}\lo R\lo M \lo 0,$$ which give us $\Tor_{i}^R (R/\fm^{n-1},M)=H^i(\ldots\lo M\stackrel{x}\lo M\stackrel{x^{n-1}}\lo M\stackrel{x}\lo\ldots), $ and so	$$
\Tor_{i}^R (R/\fm^{n-1},M)	=\left\{\begin{array}{ll}
\frac{xR}{\fm^{n-1}} &\mbox{if } i\in2\mathbb{N}+1 \\
	0  	&\mbox{otherwise } 
	\end{array}\right.
	$$In particular, $R/\fm^{n-1}$ is not tor-rigid.
	\item[(d)] Concerning (ii) $\Leftrightarrow$ (i) we may reduce the assumption from non-artinian to the case that the ring is not field.
\end{itemize}
\end{remark}
Suppose $\depth(R)=0$. It may be nice to find $\{n:R/\fm^n \emph{ is not Lichtenbaum} \}$.
This may be large, as the next example suggests.
\begin{example}\label{cl}In this item rings are of depth zero and of dimension  bigger than zero.
\begin{itemize}
	\item[(i)] 	Let $R:=k[[x,y]]/(x^2,xy)$ and let $n>0$.  Then $$R/\fm^n \emph{ is  Lichtenbaum } \Longleftrightarrow n=1 .$$
\item[(ii)] More general than i), suppose the ring is equipped with an element $x\in\fm^{\ell-1}\setminus\fm^\ell$ such that $x\fm=0$ and $\ell>0$. Then
$R/\fm^n$ is not Lichtenbaum for all $n\geq \ell$.

\item[(iii)] Let $R:=k[[x,y]]/(x^3,x^2y)$. Then $R/\fm^n$ is not Lichtenbaum for all $n\geq 3$.
\item[(iv)] More general than i), suppose  $\fm^n$ is principal for all $n\geq \ell$. Then
$R/\fm^n$ is not Lichtenbaum for all $n\geq \ell$.
\item[(v)]	Let $R:=k[[x,y]]/(x^3,x^2y,xy^2)$. Then $R/\fm^n$ is not Lichtenbaum for all $n\geq 3$.
\end{itemize} 
\end{example}

\begin{proof}  i)
Clearly, $R/ \fm$ is Lichtenbaum. For the converse part, recall that 
 $R$ is of depth zero, multiplicity one and dimension one.
In fact,
$\fm^n=(y^n)$ for all $n\geq 2$. Now, apply  the proof of Observation \ref{19} to see $R/\fm^n$ is not  Lichtenbaum, when $n>1$.

ii) See the proof of Proposition \ref{20}.

iii) This is in part (ii). Note that $x^2\fm=0$ and that $x^2\not\in \fm^3$.

iv) See the proof of (i).

v) Since $ \fm^n=(y^n)$  for all $n\geq 3$, it is enough to apply part (iv).
\end{proof}

	Let us recover a funny result of Levin-Vasconcelos:

\begin{corollary}\label{27} Let $(R,\frak{m})$  be a non-artinian local ring. Then $R$ is regular if and only if
	$R/\fm^n$ has finite injective dimension for some natural number $n$.
\end{corollary}

\begin{proof} If $R$ is regular, then every $R$-module has finite injective dimension.
	Now, assume that there exists a natural number $n$ such that $R/\fm^n$ has finite injective dimension. Immediately, Bass
	theorem implies that $R$ is Cohen-Macaulay, and  so $\depth_R R=\dim R> 0$.  Hence, $ R/\fm^n $ is a Lichtenbaum $R$-module
	by Proposition \ref{20}. Set $d:=\dim R$. Then $\id_R(R/\fm^n)=d.$
	Let $ M $ be any finitely generated $R$-module. As $ \id_R(R/\fm^n)=d $, we have $$\Ext_{R}^{1}(\Syz_{d} (M),R/\fm^n)\cong
	\Ext_{R}^{d+1} (M,R/\fm^n)=0.$$ By applying this along  from the 4-terms exact sequence

		\[\begin{array}{rl}
&	\Tor_{2}^{R}(\Tr(\Omega(\Syz_{d}(M))), R/\fm^n)\lo \Syz_{d} (M) \otimes_R R/\fm^n
	\longrightarrow \\&\Ext_{R}^{1}
	(\Syz_{d} (M), R/\fm^n)\to \Tor_{1}^{R} ( \Tr(\Omega(\Syz_{d}(M))), R/\fm^n)
	 \rightarrow 0,
	\end{array}\]
we deduce that $\Tor_{1}^{R} (\Tr(\Omega(\Syz_{d}(M))), R/\fm^n)=0$. Since
$R/\fm^n $ is   Lichtenbaum,
 $\Omega(\Syz_{d}(M))$ is also free. In other words, $ \pd_RM< \infty $. We proved every finitely generated has finite
	projective dimension. Thanks to  a theorem of Auslander-Buchsbaum-Serre, this means that $ R $ is regular.
\end{proof}Suppose
$S\subseteq F$ are finitely generated. By Artin-Rees,
there is $k>0$ such that $\fm^{n-k}(\fm^{k}F\cap S)= \fm^nF\cap S$ for all $n>k$.
Why $k>0$? In fact, the following stronger property holds:
\begin{corollary}
 Let $(R,\frak{m})$  be a local ring  such that $\depth_R R>0$  and $M$ be finitely generated which is not free. Let $F:=R^{\beta_0(M)}\to M\to 0$ be the natural map. 
 Then $\fm^n\Syz_1(M)\neq \fm^nF\cap\Syz_1(M)$  for all $n>0$.
\end{corollary}

\begin{proof}
Since $\depth(R)>0$, and in view of Proposition \ref{20}, we know $R/\fm^n$ is Lichtenbaum.
Since $M$ is not free and as $$\Tor_{1}^R (R/\fm^n ,M)= {\fm^nF\cap\Syz_1(M)}/{\fm^n\Syz_1(M)},$$
we get the desired claim.
\end{proof}

In the next section,  we will revisit Corollary \ref{27} from a different point of view. Here, we may talk  more:

\begin{proposition}\label{25} Let $ (R, \fm,k) $ be a local ring.  The following assertions are equivalent:
	\begin{itemize} 
		\item[(i)]  If  $ L $ is a finitely generated Lichtenbaum module, then $ \id_R (L) <\infty $.
		\item[(ii)] $ R $ is regular.  
		\item[(iii)]    Any finitely generated module of depth zero is Lichtenbaum.
	\end{itemize}
\end{proposition}

	\begin{proof}(i) $\Rightarrow$ (ii): 
	We know $ \id_R(L) = \depth_R (R) :=d$. From this,
	$ \Ext_{R}^{d+1} (M,L)=0$ for any finitely generated module $ M $. This in turn is equivalent with
	$ \Ext_{R}^1 (\Syz_{d} (M), L)=0 $.
	Recall that
	$$ \cdots \longrightarrow \Ext_{R}^1 (\Syz_{d} (M), L) \longrightarrow \Tor_{1} (\Tr \Omega \Syz_{d} (M),L) \longrightarrow 0,$$which yields
	$\Tor_{1} (\Tr \Omega \Syz_{d} (M),L) = 0$. Since $ L $ is Lichtenbaum and by definition, $ \Tr \Omega\Syz_{d} (M) $ is free. This yields that  $ \Omega \Syz_{d} (M) \cong\Tr(\Tr \Omega \Syz_{d} (M))$ is free. In other words, $ \pd_R(M) <\infty $ and so $ R $ is regular.

	(ii) $\Rightarrow$ (iii): This is in Corollary \ref{cr}.
	
	(iii) $\Rightarrow$ (i): First, assume that $\depth_R(R)=0$. By assumption,
	the ring $R$ is a  Lichtenbaum module. Since $\Tor^R_1(R,-)=0$ we deduce that any
	finitely generated module is free. Consequently,  $R$ is a field. So, the desired claim in this case is trivial.  Now, assume that
	 $d:=\depth_R(R)>0$. Let $\underline{x}:=x_1,\ldots,x_d$ be a maximal $R$-sequence.
	 It is well-known that $\depth_R(\Syz_d(k))=d$. Here, we need a little more, namely
$\underline{x}_j$ is 	a $\Syz_t(k)$-sequence, where  $\underline{x}_j:=x_1,\ldots,x_j$ and $1\leq j\leq  i \leq d$. Indeed,
	 we proceed by induction. Let $1\leq j\leq i$ and $\beta_i$ be the $i$-th Betti number of $R/ \fm$. We apply
	 $-\otimes_R R/\underline{x}_jR$ to the exact sequence
	 $$0\longrightarrow\Syz_{i+1}(k)\longrightarrow R^{\beta_i}\longrightarrow\Syz_i(k)\longrightarrow0,$$
	and deduce the following:$$0\longrightarrow\Tor_{1}^R(R/\underline{x}_jR,\Syz_i(k))\longrightarrow\frac{\Syz_{i+1}(k)}{\underline{x}_j\Syz_{i+1}(k)}\longrightarrow (R/\underline{x}_jR)^{\beta_i}\longrightarrow\frac{\Syz_{i}(k)}{\underline{x}_j\Syz_{i}(k)}\longrightarrow0.$$By  the induction hypothesis, we know $\underline{x}_j$
	is regular
	over $\Syz_i(k)$. We apply this to deduce that
	$\Tor_{1}^R(R/\underline{x}_jR,\Syz_i(k))=0$, and consequently,
	$${\Syz_{i+1}(k)}/{\underline{x}_j\Syz_{i+1}(k)}\subseteq (R/\underline{x}_jR)^{\beta_i}.$$
	Since
	$x_{j+1}$ is regular over $R/\underline{x}_jR$
	we conclude  that 	$x_{j+1}$ is regular over
$\frac{\Syz_{i}(k)}{\underline{x}_j\Syz_{i}(k)}$.
Therefore, $\underline{x}_{j+1}$ is 	regular over $\Syz_i(k)$.
Recall that   $\textbf{K}(\underline{x},-)$ is the Koszul complex of $(-)$ with respect to $\underline{x}$. Then
	$$\Tor^R_1(R/\underline{x}R,\Syz_d(k))=H_1\left(\textbf{K}(\underline{x},R)\otimes_R\Syz_d(k)\right)=H_1(\textbf{K}(\underline{x},\Syz_d(k)))=0.$$By
the	assumption, $R/\underline{x}R$
	is  Lichtenbaum. This yields that $\Syz_d(k)$ is free.
In other words, $\pd_R(k)<\infty$. In the light of Auslander-Buchsbaum-Serre theorem, we observe that  $R$ 	is regular.
	In particular, any (Lithtenbaum) module is of finite injective dimension. 
\end{proof}

\begin{corollary}\label{28} Let $(R,\frak{m})$  be a local ring. Assume that there exists a positive integer $n$ such that
$R/\fm^n$ has finite injective dimension.  The following assertions are valid:
\begin{itemize}
\item[(i)]    $ R $ is Gorenstein.
\item[(ii)] If $\dim R>0$, then $R$ is regular. \item[(iii)]If $\dim R=0$, then $\fm^n=0$.
\end{itemize}
\end{corollary}

\begin{proof}
	(i)  It is enough to recall from \cite{PS} that a local ring is Gorenstein if and only if it admits a nonzero cyclic
module. 

(ii) This is immediate from Corollary \ref{27}.

(iii)  By part (i) $R$ is Gorenstein.  As   $\dim R=0$, we have $ \id_R R=\depth_R R=0$, and so $R$ is injective. As $\id_R(R/\fm^n)<\infty$,
by the same reasoning, we get $ R/\fm^n $ is injective. From the short exact sequence
$$0\longrightarrow \fm^n \longrightarrow R \longrightarrow R/\fm^n \longrightarrow 0  \  \  \  (*),$$
we deduce $\id_R(\fm^n)= 0$. Thus, $(*)$ splits, and so $R=\fm^n \oplus R/\fm^n$. Multiply
both sides with $\fm^n$ yields $\fm^n=\fm^{2n}$. By Nakayama's lemma $\fm^n=0$, as claimed.
\end{proof}

\begin{remark}
Let us give another proof of Corollary \ref{28}(iii): We know $R$
is Gorenstien. Combine this with $\id_R(R/\fm^n)<\infty$ yields that  $\pd_R(R/\fm^n)<\infty$, and so 
	$R/\fm^n$ is free. Taking annihilator, we see $\fm^n=0$.
\end{remark}
The monograph \cite{AB} is our reference for the concept of  $G$-dimension.
One may simplify \cite{CW} if assumes some restrictions:

\begin{remark}\label{23}  Let $(R,\frak{m})$  be a  local ring $L$ is a finitely generated  and strong Lichtenbaum module for $R$.
	Then $\Gdim_RL<\infty$ if and only if $R$ is Gorenstein.
\end{remark}
 Also, see Corollary \ref{to}.

\begin{proof} The if part is obvious, because over Gorenstein local rings all modules have of finite $G$-dimension.	
	Conversely, suppose   $d:=\Gdim_RL<\infty$. As $\Gdim_RL=\sup \lbrace i \mid \Ext_{R}^i (L,R) \neq 0 \rbrace $,
	it follows that $\Ext_{R}^{d+1}(L,R)=0$. Now, we have
	$$0=\Ext_{R}^{d+1} (L,R)^{\vee} = \Tor_{d+1}^{R} (L, R^{\vee}) = \Tor_{d+1}^{R} (L,E) = \Tor_{1}^{R} (L,\Syz_{d}(E)).$$
	Since $ L $ is a strong Lichtenbaum  module,   $\Syz_{d}(E)$ is flat, and so $p:=\pd_R(E(R/\fm)) < \infty $. There
	is an exact sequence $$0\longrightarrow P_{p}\longrightarrow \dots \longrightarrow P_0\longrightarrow E \longrightarrow 0,$$
	in which each $P_i$ is  projective. By applying the exact functor $(-)^\vee$ on the above exact
	sequence, we see that $R\cong E^\vee$ has finite injective dimension. Thus, $ \hat{R}$ is Gorenstein, and so $R $ is as well.
\end{proof}

Here, is the lifting property of Lichtenbaum modules:

\begin{observation}
Let $(R,\fm)$ be a  local ring, $x\in\fm$ and $L$ be a finitely generated module
such that
$xL=0$. Suppose $L$
is a Lichtenbaum module over $\overline{R}:=R/xR$. Then $L$ is Lichtenbaum as an $R$-module.
\end{observation}

\begin{proof}
Recall that $L=\frac{L}{xL}$, and it is indeed equipped with a structure of an $R/xR$-module.
Let $F$ be a finitely generated $R$-module such that $\Tor_{1}^R(L,F)=0$.
By applying the standard argument we know that $\Tor_{1}^{\overline{R}}(L,F/xF)=0$. Since $L$ is Lichtenbaum  over $\overline{R}$, we deduce that $\overline{F}:=F/xF$ is free as an $\overline{R}$-module. Say $\overline{F}=\bigoplus_n{\overline{R}}$. Let $k:=R/\fm= \overline{R}/ \fm\overline{R}$. Thanks to \cite[Theorem 2.3]{Mat} we know
				$$n=\ell \left(\frac{\overline{F}}{\fm\overline{F}}\right)=\ell \left(\frac{F}{\fm F}\right)=\mu_R(F).$$Then there is an exact sequence $0\to \Syz_1(F)\to R^n\to F\to  0$. Tensor it with $-\otimes_R\overline{R}$ we deduce the following $$\Syz_1(F)\otimes_R\overline{R}\lo \overline{R}^n\stackrel{\phi}\lo\overline{F}\lo 0,$$
				where $\phi$ is an isomorphism. It turns out that $\Syz_1(F)\otimes_R\overline{R}=0$.
				In view of Nakayama's lemma, $\Syz_1(F)=0$ (see \cite[Ex. 2.3]{at}). By definition, $F=R^n$ which is free. So,  $L$ is Lichtenbaum when we view it as an $R$-module.
\end{proof}

\section{Connections to a result of Burch}

\begin{definition}
	A  nonzero $R$-module $L$ is called quasi Lichtenbaum if for every  finitely generated $R$-module $N$, the vanishing of
	$\Tor^R_1(L,F)=\Tor^R_2(L,F)=0$ implies that $F$ is  free.
\end{definition}
As a source of quasi  Lichtenbaum modules, we recall the concept of Burch ideals.
An ideal $I$ is called Burch, if $\fm( I :_R \fm) \neq I  \fm  $.

\begin{example}
	There is a quasi-Lichtenbaum module which is not  Lichtenbaum.  
\end{example}

\begin{proof}
	Let $R:=k[[x,y]]/(x^2,xy)$.   In order to see  $I:=y^2R$
	is Burch, recall that  $(I:\fm)=\fm$ and so  $$\fm(I:\fm) =\fm^2=y^2R\neq y^3R=I\fm.$$Following  definition,
	$I$
	is Burch. Thanks to Observation \ref{b} (see below), we deduce that  $R/I$ is quasi-Lichtenbaum.
	Since $\Tor_{1}^R(R/I,R/xR)=0$ and $R/xR$ is   not free, we conclude that $R/I$ is  not Lichtenbaum.
\end{proof}


\begin{proposition}
Let $(R,\frak{m})$ be a local ring and let $L$ be finitely generated  such that $\pd_R(L)< \infty$ and
$$\Tor^{R}_{1}(L,M)=\Tor^{R}_{2}(L,M)=0\Longrightarrow\pd_R(M)\leq 1,$$where $M$ is finitely generated.
	Then $\depth_R(L)\leq 1$. 
\end{proposition}

\begin{proof}
	On the contrary, assume that $\depth_RL>1$. 
	Then there is an $L$-regular sequence $\underline{x}:=x,y\in \frak{m}$.
	According to Auslander's zero-divisor, we know $\underline{x}$ is an $R$-sequence.
From this, we compute the following homologies $$\Tor^{R}_{i}(L,R/ \underline{x}R)=H_i(\textbf{K}(\underline{x},R)\otimes_RL)=H_i(\textbf{K}(\underline{x},L))=0,$$
	where $1\leq i\leq 2$.
By the assumption,  we get to a contradiction.
\end{proof} 

\begin{example}
Let $R:=k[[x]]$ and apply the previous result for $L:=\fm$. This shows that the bound $\depth_L(L)\leq 1$  achieves.\end{example}

\begin{corollary}\label{cr1}
	Let $(R,\fm)$ be  hypersurface,   $L$ be finitely generated and $\depth_R(L)=0$.
	Suppose	$\Tor^R_1(L,M)=\Tor^R_2(L,M)=0$ where $M$ is finitely generated. Then $M$ is maximal Cohen-Macaulay.
\end{corollary}

\begin{proof}
	By  a result of Murthy \cite[Theorem 1.6]{MUR}, we observe that $\Tor_+^R(L,M)=0$. By Huneke-Wiegand \cite[Theorem 1.9]{HW1}, either $\pd(L)<\infty$ or $\pd(M)<\infty$. This allow us to apply depth formula.
	By depth formula, $$\depth_R(M)=\depth_R(L)+\depth_R(M)=\depth_R(L\otimes M)+\depth_R(R)\geq \depth(R)=\dim R,$$
	as claimed.
\end{proof}

\begin{notation}
	By $\overline{R}$ we mean the integral closure of a domain  $R$ in its  field of fractions.
\end{notation}

\begin{proposition}
	Let $(R,\frak{m})$ be a $1$-dimensional  local integral domain and let $M$ be        a finitely generated module such that  
	$$\Tor^{R}_{1}(\overline{R},M)=\Tor^{R}_{2}(\overline{R},M)=0\quad(\ast)$$
	Then $\pd_R(M)<2$. 
	\end{proposition}

\begin{proof}
Let $\mathbf{F}:= \ldots \lo F_1 \lo F_0\lo M \lo 0$ be a free resolution of $M$.
Tensor  it with $-\otimes_R\overline{R}$, yields the following complex$$ F_3\otimes_R\overline{R}\lo F_2\otimes_R\overline{R}\lo F_1\otimes_R\overline{R}\lo F_0\otimes_R\overline{R}\lo M\otimes_R\overline{R}\lo 0,$$which is exact, because we have the vanishing property $(\ast)$, and recall that the above complex can be used for a part of a free resolution of  the $\overline{R}$-module $M\otimes_R\overline{R}$. Now, tensor this with $-\otimes_{\overline{R}} \overline{R}/\fm\overline{R}$, yields the following diagram of  complexes:

	\[
\begin{CD}
\zeta_1:= (F_3\otimes_R\overline{R})\otimes_{\overline{R}} \frac{\overline{R}}{\fm\overline{R}} @>>>  \ldots@>>> (F_0\otimes_R\overline{R})\otimes_{\overline{R}} \frac{\overline{R}}{\fm\overline{R}}	  @>>>  (M\otimes_R\overline{R})\otimes_{\overline{R}}\frac{\overline{R}}{\fm\overline{R}} @>>> 0\\
@V \cong VV  @V \cong VV   @VV\cong V  @VV\cong V \\
\zeta_2:=F_3\otimes_R  \overline{R}/\fm\overline{R} @> >> \ldots@> >> F_0\otimes_R  \overline{R}/\fm\overline{R} @>>> M\otimes_R  \overline{R}/\fm\overline{R} @>>> 0.
\end{CD}
\]
By this identification, the corresponding homology of $\zeta_1$ and $\zeta_2$ at the second spot
coincides, i.e., the following holds:$$\Tor^{\overline{R}}_{2}(M\otimes_R\overline{R},\frac{\overline{R}}{\fm\overline{R}})=H^2(\zeta_1)\cong H^2(\zeta_2)=\Tor^{R}_{2}(M,\frac{\overline{R}}{\fm\overline{R}}),$$when we view them as $R$-modules. Recall that the integral closure of a noetherian domain of
dimension one is noetherian.  It turns out that $\overline{R}$ is a Dedekind domain.
Since $\overline{R}$ is a Dedekind domain, its global dimension is one. This yields that $\Tor^{\overline{R}}_{2}(M\otimes_R\overline{R},\frac{\overline{R}}{\fm\overline{R}})=0$. Now recall
that $\frac{\overline{R}}{\fm\overline{R}}=\bigoplus R/ \fm$. We apply this at the previous
displayed item
to conclude that $$\bigoplus\Tor^{R}_{2}(M,R/ \fm)=\Tor^{R}_{2}(M,\frac{\overline{R}}{\fm\overline{R}})=0.$$In other words, $\Tor^{R}_{2}(M,R/ \fm)=0$. Since  $M$ is   finitely generated, 
$\pd_R(M)\leq1$.
\end{proof}

\begin{corollary}\label{+tor}
	Let $(R,\frak{m})$ be a $1$-dimensional  local integral domain and let $M$ be        a finitely generated module such that  
	$$\Tor^{R}_{i}(\overline{R},M)=\Tor^{R}_{i+1}(\overline{R},M)=0$$for some $i>0$. Then
$\pd_R(M)\leq1$. In particular,   $M$ is free provided it is of positive depth. 
	\end{corollary}

\begin{proof}
Without loss of generality, we may assume that $i=1$.
Indeed, let $j:=i-1$. We pass to $M':=\Syz_j(M)$, the $j$-th syzygy module of $M$, and recall that $\pd(M')\leq j+\pd(M)$.   Let us apply the previous result to see $\pd_R(M')\leq1$. So, $\pd_R(M)\leq i$. In the light of Auslander-Buchsbaum formula we observe that 
$\pd_R(M)\leq1$.
The particular case follows by Auslander-Buchsbaum formula.
\end{proof}

\begin{problem}
Find conditions for which $\overline{R}$ is tor-rigid.\end{problem}
Here, is a partial positive  answer:
\begin{corollary}\label{cr2}
	Let $(R,\fm)$ be a  $1$-dimensional complete local integral domain which is hypersurface.
	Let $M$ be        a finite length module such that  
	 $\Tor^{R}_{i}(\overline{R},M)=0 $ for some $i>1$. Then
	$\pd_R(M)\leq1$. 
\end{corollary}

\begin{proof}The integral closure of a noetherian domain of
dimension one is  noetherian, but
not necessarily module-finite. The complete assumption implies that $\overline{R}$ is module-finite.
	By  a result of Huneke-Wiegand \cite[2.3 Corollary]{HW}, we observe that $\Tor_{\geq i}^R(\overline{R},M)=0$. Now, apply Corollary \ref{+tor}.
\end{proof}

\begin{example}
The dimension restriction is needed. Indeed, let $R:=k[[x^2,xy,y^2]]=\frac{k[[u,v,w]]}{(uv-w^2)}$. This is a $2$-dimensional
normal hypersurface integral domain. In particular, 	 $\Tor^{R}_{i}(\overline{R},M)=0 $ for all $i>0$ and all modules $M$. \end{example}
The next result when $\ell(M)<\infty$ implicitly is in \cite[Theorem 4.7]{h}:
\begin{fact}\label{cr2a}
	Let $(R,\fm)$ be a  $1$-dimensional complete local integral domain of prime characteristic  with algebraically closed residue field.
	Let $M$ be        a finitely generated module such that  
	$\Tor^{R}_{i}(\overline{R},M)=0 $ for some $i>1$. Then
	$\pd_R(M)\leq1$. 
\end{fact}

\begin{proof}Recall that there is an $n$, large enough, so that $\oplus\overline{R}= \up{\varphi^n}R$ (see \cite[Lemma 4.5]{h}). In particular, $\Tor^{R}_{i}(\up{\varphi^n}R,M)=0 $ for some $i>1$. But,
	$\up{\varphi^n}R$ is tor-rigid over Cohen-Macaulay rings of dimension one, see \cite[2.2.12]{mi}. So, 		$\pd_R(M)\leq1$. 
\end{proof}
Here, is a partial negative answer:

\begin{observation}\label{nt}
	Let $(R,\fm,k)$ be a $2$-dimensional complete local domain which is not Cohen-Macaulay (for example $R=k[[x^4,y^4,x^3y,xy^3]]$). Then $\overline{R}$ is not  tor-rigid.
	\end{observation}
The reader  may skip parentheses if not interested in the example.

\begin{proof}
Recall that $\overline{R}$ is finitely generated as an $R$-module, and also it is noetherian and local, because $R$ is complete  (for example $\overline{R}=k[[x^4,y^4,x^3y,xy^3,x^2y^2]]$. Note that $x^2y^2=\frac{(x^3y)^2}{x^4}$ belongs to the fraction field of $R$, and it is the root of $f(T)=T^3-x^4T\in R[T]$. Since $k[[x^4,y^4,x^3y,xy^3,x^2y^2]]$ is the invariant ring, it is normal, and so becomes the integral closure of $R$). Suppose on the way of contradiction that $\overline{R}$ is   tor-rigid. This allows us to  apply \cite[4.3]{Au} to deduce that  each  $\overline{R}$-regular sequence is an  ${R}$-regular sequence. Now, let $a,b$
be a system of parameter for $R$, and so a parameter sequence for $\overline{R}$ (for example, in the example set $a:=x^4, b:=y^4$). By
Serre's characterization of normality, see \cite[Theorem 23.8]{Mat}, we know $\overline{R}$ satisfies Serre's condition $(S_2)$  (for the definition, see \cite[page 183]{Mat}). Thus, $a,b$ is an $\overline{R}$-regular sequence.
By the mentioned result of Auslander, $a,b$ is an ${R}$-regular sequence. So,
$R$ is Cohen-Macaulay, a contradiction.
\end{proof}
\begin{remark}
 It may be nice to give situations for which 	$\pd_R(\overline{R})=\infty$ or even $\Gdim_R(\overline{R})=\infty$.
 
 \begin{enumerate}
 	 \item[(i)] 	Let $(R,\fm)$ be a $1$-dimensional local domain which is not regular.
 	Then 	$\pd_R(\overline{R})=\infty$.
 	\item[(ii)]  Let $(R,\fm)$ be a $2$-dimensional Cohen-Macaulay complete domain which is not normal.
 	Then 	$\pd_R(\overline{R})=\infty$.
 	
 	\item[(iii) ] Let $(R,\fm,k)$ be a $1$-dimensional complete local
 	domain of prime characteristic with $k=\overline{k}$. If $\Gdim_R(\overline{R})<\infty$, then $R$ is Gorenstein.
 		\item[(iv)]  Let $(R,\fm)$ be a $2$-dimensional Cohen-Macaulay complete domain which is not   quasi-normal.
 	Then   $\Gdim_R(\overline{R})=\infty$.
 \end{enumerate}
\end{remark}

\begin{proof}
	(i): This is similar to ii).
	
	(ii): Recall that $\overline{R}$ is noetherian and local. By definition, its normal. Suppose on th way of contradiction that
	$\pd_R(\overline{R})<\infty$.
	Recall that $R$ and $\overline{R}$ both are Cohen-Macaulay. In particular, both of them satisfy Serre's condition $(S_2)$. Due to Auslander-Buchsbaum formula we know $R\to \overline{R}$ is flat (in fact free). By
	Serre's characterization of normality (see \cite[Theorem 23.8]{Mat}), we know $R$ is not $(R_1)$ and $ \overline{R}$ is 
	$(R_1)$. This is in contradiction with \cite[Theorem 23.9]{Mat}. So, $\pd_R(\overline{R})=\infty$.
	
	(iii): By the argument presented in \cite[Theorem 4.7]{h} we know that $\Gdim_R(\up{\varphi^n}R)<\infty$
	for some $n$ large enough. Thanks to Auslander-Bridger formula, $\Gdim_R(\up{\varphi^n}R)=0$.
	In particular, $\Ext_{R}^1(\up{\varphi^n}R,R)=0$. This implies that $R$ is Gorenstein, see e.g. \cite[Corollary 2.7]{l}. 
	
	(iv): Recall that quasi-normal means $(S_2)$+$(G_1)$, here a ring $A$ is called $(G_1)$
	if $A_{\fp}$ is Gorenstein for all prime ideal $\fp$ of height at most one.
	Now, the desired conclusion is a slight modification of part (ii), and we leave   details to the reader.
\end{proof}

\begin{proposition}
	Let $(R,\frak{m})$ be a local integral domain and let $L$ be such that  
	$$\Tor^{R}_{1}(L,M)=\Tor^{R}_{2}(L,M)=0\Longrightarrow\emph{M is torsion-free}.$$
	Then $\depth_R(L)= 0$. 	\end{proposition}
One may replace the integral domain assumption with $\pd_R(L)<\infty$, and derives the same conclusion.
\begin{proof}
	If not,  then  $\depth_RL>0$, i.e.,  there is an $L$-regular element $x$.
It is easy to 	see, $\Tor^R_2(L,R/xR)=\Tor^R_1(L,R/xR)=0$.  Since  $R/xR$ is not torsion-free, we get a contradiction.
\end{proof}

 Several years ago, Burch proved:

\begin{theorem}
 Let $M$ be  finitely generated and $I$ be Burch. If $\Tor^R_t
	(R/I,M) =
	\Tor^R _{t+1}(R/I,M) = 0$ for some positive integer $t$, then $\pd_R(M)\leq t$.
\end{theorem}
Despite its importance, this was proved very recently in \cite[Theorem 1.2]{Dey} that
the above bound is not sharp. Our elementary and short proof is independent
of them:
\begin{observation}\label{b}
Adopt the above notation. Then 	$\pd_R(M)\leq t-1$.
	\end{observation}

\begin{proof}
This is a combination of Fact \ref{lemma1} and the above result of Burch.
\end{proof}

In fact \cite[Theorem 1.2]{Dey}
presents the module version of Observation \ref{b}. By using some ideas taken from  \S3 we observed in the previous draft that if a Burch ideal
is of finite injective dimension, then the ring is regular.
 Also, this result is in  \cite[Corollary 3.20]{Dey}. Since their argument is short compared to us, 
 we skip the mentioned observation. Instead, we apply it for $\fm M$ and reconstruct some known results:
 
 \begin{corollary}(Levin-Vasconcelos) Let $M$ be finitely generated and such that $\fm M$ is nonzero
 	and of  finite injective (projective) dimension. Then $R$ is regular.
 \end{corollary}

 \begin{example}
The	finitely generated assumption is needed. For example, let $R$ be any integral domain which is not regular. Since $E_R(k)$ is divisible, we have $\fm E_R(k)=E_R(k)$. Then $\fm E_R(k)$ is nonzero and injective. But,  $R$ is not regular.
\end{example}

The following  extends the main result of \cite{tony} by dropping an assumption on Castelnuovo-Mumford regularity \footnote{we only focus on $\CI_R(-)$ (see \cite{AGP} for its definition and basic properties) and other homological invariants such as $\Gdim_R(-)$ follow in the same vein.}:

\begin{corollary}\label{to}
Let $(R,\fm, k)$ be any local ring,  $M$ be  finitely generated of positive depth such that $\fm^i M$ is nonzero
and of  finite  complete-intersection dimension for some $i>0$. Then $R$ is complete-intersection.
\end{corollary} 

Instead of $\fm M$ one may assume the module is Burch and derives the same conclusion.

\begin{proof}
Let $x\in\fm$ be regular over $M$. Clearly, $x$ is $\fm^i M$-sequence. We know from \cite{AGP} that
$$\CI_R(\frac{\fm^i M}{x \fm^{i} M})=\CI_R(\fm ^iM)-1<\infty.$$
By \cite[Lemma 3.7]{Dey}, $k$ is a direct summand of $\frac{\fm^i M}{x \fm^{i } M}$.  This yields that 
$\CI_R(k)<\infty$, and so $R$ is  complete-intersection.
\end{proof}

\section{An application: dimension of syzygies}

Rings in this section are not artinian.  Recall that $\ell(-)$ is the length function. 
\begin{question}(See \cite[Question 1.2]{h})
	Let $M$ be such that $\pd_R(M) = \infty$ and $\ell(M) <\infty$. Is $\ell(\Syz_i(M))=\infty$ for all $i > \dim(R) + 1$?
\end{question}

\begin{proposition}\label{bl}	Let $L$ be Lichtenbaum, $\pd_R(L) = \infty$ and $\ell(L) <\infty$.
 Then $\ell(\Syz_i(L))=\infty$ for all $i > 0$.
\end{proposition}

\begin{proof}
	By the short exact sequence $0\to \Syz_1(L)\to R^n\to L\to 0$
	and the facts that $\ell(L)<\infty=\ell(R^n)$
	we see $\ell(\Syz_1(L))=\infty$. Suppose $\ell(\Syz_{i+1}(L))<\infty$
	for some fixed $i > 0$. In the light of \cite[Lemma 4.1]{A} we observe
	$$\Tor^R_1(L,\Syz_{i-1}(R/ H^0_{\fm}(R)))=\Tor^R_i(L,R/ H^0_{\fm}(R))= 0.$$  
	Since
	$M$ is Lichtenbaum, $\Syz_{i-1}(R/ H^0_{\fm}(R))$ is free.
	In other words, $\pd_R(H^0_{\fm}(R))<\infty$.
	Suppose $H^0_{\fm}(R)\neq0$. According to a  celebrated result of Burch, 
	$H^0_{\fm}(R)$ contains a regular element $x$. By definition,
	$\fm^n x=0$ for some $n>0$. Since $x$ is regular, we observe that 
	$H^0_{\fm}(R)=0$, i.e., $\depth_R(R)>0$.
	In view of \cite[Lemma 3.4]{A}, we see 
	$\ell(\Syz_i(L))=\infty$.
	This completes the proof.
	\end{proof}

\begin{remark}\label{lf}Adopt the notation of Proposition \ref{bl}. Instead of  $\ell(L) <\infty$ we may assume that $L$ is locally free on the punctured spectrum. By the same proof
	$\ell(\Syz_i(L))=\infty$ for all $i > 1$.
\end{remark}

We say a local ring	$R$ is of isolated  singularity if it is singular and $R_{\fp}$ is regular for all $\fp\in(\Spec(R)\setminus\{\fm\})$.
\begin{corollary}\label{obfr}
	Let $(R,\fm)$ be a complete local ring of prime characteristic with perfect residue field.  If
	$R$ is of isolated  singularity, then
	 $\ell(\Syz_i(\up{\varphi^n}R ))=\infty$ for all $i > 1$ and all $n\gg 0$. 
\end{corollary}

\begin{proof}
	By the assumption, $R$ is $F$-finite (see \cite[Page 398]{BH}). This means that $\up{\varphi^n}R$ is finitely generated as an $R$-module.
	By a celebrated theorem of Kunz (see \cite[Corollary 8.2.8]{BH}) we know that $\pd_R(\up{\varphi^n}R) = \infty$, because $R$ is not regular. Recall from  \cite[Proposition 8.2.5]{BH}     that $(\up{\varphi^n}R)_{\fp}=\up{\varphi^n} (R_{\fp})$    for all $\fp\in\Spec(R)\setminus\{\fm\}$. 
	Again, another use of \cite[Corollary 8.2.8]{BH}  shows that $\up{\varphi^n}R$ is locally free on the punctured spectrum, because $R_{\fp}$ is regular for all $\fp\in\Spec(R)\setminus\{\fm\}$. Without loss of generality we may assume that $\depth(R)=0$, see the proof of \cite[Lemma 3.4]{A}. In view of Remark \ref{fr}
we observe that 	$\up{\varphi^n}R$ is Lichtenbaum for all $n\gg 0$.
	So, we are in the situation of Remark \ref{lf}. We conclude from this remark that 	$\ell(\Syz_i(\up{\varphi^n}R ))=\infty$ for all $i > 1$.
	\end{proof}

The  Lichtenbaum  assumption in  Proposition \ref{bl}       is important:

\begin{example}
	Let $R:=k[[x,y]]/(x^2,xy)$. Since  $R\stackrel{x}\lo	R\stackrel{y}\lo R\to R/yR\to0$ is exact, $\Syz_2(R/yR)=xR=R/(0:x)=k$
is of length one. In order to see  $R/yR$
is not Lichtenbaum, recall that $\Tor^R_1(R/yR,R/xR)=0$ and that $R/xR$ is not free.
\end{example}
In the previous example $\fm(y:\fm) =\fm^2=y^2R=y\fm$. So, $yR$ is not Burch. 
One may ask: Is Proposition \ref{bl} true for Burch modules?
This is not the case as the next example indicates:

\begin{example}
	Let $R:=k[[x,y]]/(x^2,xy)$. Since  $R\stackrel{x}\lo	R\stackrel{y^2}\lo R\to R/y^2R\to0$ is exact, $\Syz_2(R/y^2R)=xR=R/(0:x)=k$
	is of length one, and recall that  $I:=y^2R$
	is Burch.
\end{example}

In the previous two examples
$\Tor^R_1(R/I,A)=0$ where $I$ is generated  by an $A$-regular sequence. Is this true in general?
In fact, this was asked before than us:

\section{Vanishing and non-vanishing of  $\Tor_1$}

\begin{question}\label{61}Let $ (R,\fm)$ be a  local ring and $M$ a finitely generated $R$-module. Let $x_1,\ldots,x_t$ be an $M$-regular sequence and $I=(x_1,\ldots,x_t)$. Is it true that
$\Tor^R_1(R/I^n,M)=0$
for all $n\geq1$?\end{question}

\begin{observation}Here, we collect a couple elementary observations:
	\begin{enumerate} \item[(i)] 	The case M is cyclic is in \cite[Ex. 11.12]{ht}.

\item[(ii)]  If $I$ is generated by regular sequence, the answer is positive. Indeed, we apply a routine induction
on $n$, to reduce to the case $n=1$. In this case, we have $$\Tor^R_1(R/I,M)=H_1(\textbf{K}(\underline{x},R)\otimes_RM)=H_1(\textbf{K}(\underline{x},M))=0.$$So, the claim follows.

\item[(iii) ]  If $\pd_R(M)$ is finite, the answer is positive. Indeed, in view of  Auslander's zero divisor theorem, $I$ is generated by a regular sequence $\underline{x}$. In view of ii)
we see $\Tor^R_1(R/I^n ,M)=0$  for all $n\geq1$.
 \item[(iv) ]The question is true over regular rings. Indeed, as over regular local rings modules are of finite projective dimension, the claim is in iii).
 \item[(v) ] The question is true if $I$ is principal. 
 \item[(vi) ] The question is true  over 1-dimensional rings. Indeed,
recall that $$1\leq \grade_R(I,M)\leq\depth_R(M)\leq\dim(M)\leq \dim(R)=1.$$ This says that $ \grade(I,M)=1$.
Since $I$ is generated by an $M$-sequence, we deduce that $I$ is principal. Now the desired claim is in the previous item.
	\end{enumerate}
\end{observation}
Here, is a nontrivial example:

\begin{example}\label{e1}
Let $R:=k[[x^4,y^4,x^3y,xy^3]]$ and $I:=(x^4,y^4)$.  The following holds:	
	\begin{enumerate} \item[(a)] $B:=k[[x^4,y^4,x^3y,xy^3, x^2y^2]]$ is finitely generated
		as an $R$-module.
		\item[(b)] $x^4,y^4$ is   $B$-regular.
		\item[(c)]  $\Tor^R_1(R/I ,B)=0$.
	\end{enumerate}
\end{example}

\begin{proof}
	In view of \cite[42.4]{equ} we observe that $\underline{x}:=x^4,y^4$ is   $B$-sequence and there
	is an exact sequence $$\zeta:=0\longrightarrow R\longrightarrow B\longrightarrow k\longrightarrow 0.$$
	We apply $-\otimes_RR/\underline{x}R$ to  $\zeta$ and deduce the following exact sequence
	$$0\to \Tor^R_1(R/I ,B)\to \Tor^R_1(R/I ,k)\to R/I\otimes_RR\to R/I\otimes_R B\to R/I\otimes_R k\to 0.$$
	Since $2=\mu(I)=\beta_1(R/I)=\dim_k(\Tor^R_1(R/I ,k))$ we see $\Tor^R_1(R/I ,k)=k\oplus k$. Also, $R/I\otimes_R k=R/\fm+I=k$. Let $\ell:=\ell(\Tor^R_1(R/I ,B))$.  
By plugging these 
	in the previous sequence and taking
	the length, we obtain $$\ell=1-\ell( R/I)+\ell( B/IB)\quad(\ast)$$On the one hand we have
	
	$$0=\frac{IB}{IB}\subsetneqq\frac{IB+(x^3y)B}{IB}\subsetneqq\frac{IB+(x^3y,xy^3)}{IB}\subsetneqq\frac{\fm_B}{IB}\subsetneqq B/ I B.$$
	In order to see the factors are simple modules, we remark that    $\fm xy^3\in I B$ and  $\fm x^3y\in I $.
	Indeed, for example we have $( x^3y) x^3y=x^6y^2$ and $x^2y^2 \in  B$,  i.e.,  $( x^3y)  x^3y =(x^2y^2)x^4\in I B$.
It turns out that $\ell(  B/IB)= 4$. 
	On the other hand, we have
	the following composite sequence $$0=\frac{I}{I}\subsetneqq\frac{I+(x^3y)}{I}\subsetneqq\frac{I+(x^3y,xy^3)}{I}=\frac{\fm_R}{I}\subsetneqq R/ I.$$In order to see the compute the factors, we remark that    $\fm x^3y\notin I $ since
	 $x^2y^2 \notin  R$. However, the factors of $0=\frac{I}{I}\subsetneqq\frac{I+(x^6y^2)}{I}\subsetneqq\frac{I+(x^3y)}{I}$ are simple.
	In the same vein, the factor of $$\frac{IB+(x^3y)B}{IB}\subsetneqq\frac{IB+(x^3y,xy^3)}{IB}.$$
	is of length two.
We proved  that $\ell( R/I)=5$.  In view of $(\ast)$ we have $\ell =0$. In other words, $\Tor^R_1(R/I ,B)=0$, as claimed.
\end{proof}

\begin{lemma}(See \cite[Lemma 2.6]{st})
	Let $F$ be an R-module, K be a submodule of F and set $M = F/K$. Let $J = (a_1, \ldots , a_r)$ be an ideal
	generated by an $M$-regular sequence. Then $J^n F \cap K = J^nK$
	 for all $n > 0$.
\end{lemma}

\begin{observation}\label{qs}
Answer  to Question \ref{61} is yes.
\end{observation}
\begin{proof} We look at $0\to \Syz_{1}(M) \to F:=R^n\to 	M\to 0 \quad(\ast)$, i.e.
 $M\cong F/ \Syz_{1}(M)$. Let $I:= (a_1, \ldots , a_r)$. 
 We tensor $(\ast)$ with $R/I$ gives us $$0=\Tor_1^R(F, R/I) \lo \Tor_1^R(M, R/I) \lo \Syz_{1}(M)/I\Syz_{1}(M) \lo F/IF \lo  M/IM \lo 0.$$
 Then, $\Tor_1^R(M, R/I) $ is isomorphic to $(\Syz_{1}(M) \cap IF)/I\Syz_{1}(M)$. 
Combining these with the previous lemma,  by plugging $K:=\Syz_{1}(M)$, it follows that if $a_1,\ldots,a_r$ is an $M$-regular sequence, then $\Tor_1^R(M, R/(a_1,\ldots,a_r)^n)=0$ for all $n>0$. \end{proof}

\section*{Acknowledgements}
The authors are grateful to 
Souvik Dey for his comments on the paper in particular on item \ref{qs}.



\begin{thebibliography}{99}
	
	\bibitem{at} M. F.
Atiyah and    I. G.  Macdonald, {\it Introduction to commutative algebra}. Addison-Wesley Publishing Co., Reading, Mass.-London-Don Mills, Ont. 1969.

	\bibitem{tony}
J. Asadollahi; T. J. Puthenpurakal, {\it An analogue of a theorem due to Levin and Vasconcelos}, Commutative Algebra
 and Algebraic Geometry, 9–15, Contemp. Math., 390, Amer. Math. Soc., Providence, 2005.
 
\bibitem{AB}
Maurice Auslander and Mark Bridger, \emph{Stable module theory}, Mem. of
the AMS  {\bf94}, Amer. Math. Soc., Providence 1969.

\bibitem{Au}{M. Auslander}, {\it Modules over unramified regular local rings}, Illinois J. Math.  {\bf5} (1961), 631-647.

\bibitem{A}
M. Asgharzadeh,  {\it On the dimension of syzygies}, arXiv:1705.04952.

\bibitem{A1}
M. Asgharzadeh,  {\it A note on Cohen-Macaulay descent}, arXiv:2011.04525.
\bibitem{AGP}
Luchezar L. Avramov, Vesselin N. Gasharov and Irena V. Peeva, \emph{Complete intersection dimension}, Publ. Math. I.H.E.S. {\bf 86} (1997), 67--114.


\bibitem{BH}{W. Bruns and J. Herzog}, {\it Cohen-Macaulay rings},  Cambridge Studies in Advanced Mathematics,
{\bf 39},
Cambridge University Press, Cambridge, 1993.


\bibitem{B}
L. Burch, {\it On ideals of finite homological dimension in local rings}, Proc. Cambridge Philos. Soc. {{\bf64}} (1968), 941-948.


\bibitem{c1}
O. Celikbas, H. Dao, R. Takahashi, {\it Modules that detect finite homological dimensions},
Kyoto J. Math. {{\bf54}}(2), 295-310 (2014).

\bibitem{CW}{O. Celikbas; S. Sather-Wagstaff}, {\it
Testing for the Gorenstein property},
Collect. Math. {{\bf67}} (2016), no. 3, 555-568.


\bibitem{CK}{O. Celikbas and T. Kobayashi}, {\it On a class of Burch ideals and a conjecture of Huneke and Wiegand}, Collect. Math. {{\bf 73}} (2022), no. 2, 221-236.

\bibitem{CT}{O. Celikbas and R. Takahashi}, {\it Powers of the maximal ideal and vanishing of (co)homology}, Glasg. Math. J. {\bf 63} (2021), no. 1, 1–5. 


\bibitem{h}
A. De Stefani, C. Huneke, and L. Nunez-Betancourt, \emph{Frobenius Betti numbers and modules of finite
	projective dimension}, Journal of commutative algebra {\bf{9}} (4), (2017) 455-490.



\bibitem{HW}
Craig Huneke and Roger Wiegand, \emph{Tensor products of modules and the rigidity of Tor}, Math. Ann.  {\bf299} (1994), 449-476.
\bibitem{HW1}
Craig Huneke and Roger Wiegand, \emph{Tensor products of modules, rigidity and
	local cohomology}, Math. Scand.  {\bf81} no.2, (1997), 161-183.

\bibitem{ht} Craig Huneke, {\it Tight closure and its applications}. With an appendix by Melvin Hochster. CBMS Regional Conference Series in Mathematics, 88. Published for the Conference Board of the Mathematical Sciences, Washington, DC; by the American Mathematical Society, Providence, RI, 1996.

\bibitem{Dey}
Souvik Dey and Toshinori Kobayashi, {\it Vanishing of (co)homology of Burch and related submodules}, Illinois J. Math. 67(1), (2023),  101-151.

\bibitem{equ} M. 
Herrmann, S. Ikeda and U. Orbanz,  {\it
Equimultiplicity and blowing up},
An algebraic study. With an appendix by B. Moonen. Springer-Verlag, Berlin, 1988.


\bibitem{laz}D. Lazard, {\it Autour de la platitude}, Bull. Soc. Math. France {\bf97} (1969), 81-128.

\bibitem{LV}{G. Levin and V. Vasconcels}, {\it Homological dimensions and Macaulay rings}, Pacific Journal of Math.,
{\bf 25}(2), (1968), 315-323.


\bibitem{l}J.
Li, {\it Frobenius criteria of freeness and Gorensteinness}, Arch. Math. (Basel) {\bf 98} (2012), no. 6, 499-506.


\bibitem{L}
Stephen Lichtenbaum, \emph{On the vanishing of Tor in regular local rings}, Ill. J. Math.  {\bf10} (1966),
220-226.

\bibitem{Mat}
H. Matsumura, \emph{Commutative Ring Theory}, Cambridge Studies in Advanced Math, \textbf{8}, (1986).


 
\bibitem{mi}
Claudia Miller, \emph{The Frobenius endomorphism and homological dimensions}, Commutative Algebra (Grenoble/Lyon, 2001), Contemp. Math., vol.  {\bf331}, pages 207-234, Amer. Math. Soc., Providence, RI, 2003.

\bibitem{MUR} M. Pavaman
 Murthy, \emph{Modules over regular local rings}, Illinois J. Math. {\bf7} (1963), 558-565.
 

\bibitem{quy} P. Quy, \emph{Vanishing of tor},
https://mathoverflow.net/questions/118475/vanishing-of-tor. 
Also, see  the following web-page: 
\emph{Ideals generated by regular sequences and the vanishing of  Tor}, in: https://math.stackexchange.com/questions/268542/ideals-generated-by-regular-sequences-and-the-vanishing-of-operatornametor



\bibitem{st}Janet Striuli, \emph{A uniform Artin–Rees property for syzygies in rings of dimension
	one and two},
Journal of Pure and Applied Algebra {\bf{210}} (2007) 577-588.




\bibitem{int}
P. Roberts, \emph{Le th\'{e}or\`{e}me d' intersection}, C. R. Acad. Sci. Paris Ser. I Math.,
{\bf{304}}  (1987),   177-180.


\bibitem{PS}
Christian Peskine and Lucien Szpiro, \emph{Dimension projective finie et cohomologie locale},
Publ. Math. IHES.  {\bf42} (1973),  47--119.

\bibitem{v}
Wolmer V. Vasconcelos,  \emph{ On the homology of $I/I^2$}, Comm. Algebra {\bf{6}} (1978), no. 17, 1801-1809.
\end{thebibliography}
\end{document}